\documentclass[reqno,a4paper]{amsart}

\usepackage{mathrsfs,amsmath,amssymb,graphicx,fancyhdr}
\usepackage[percent]{overpic}
\usepackage[foot]{amsaddr}
\usepackage{stmaryrd}
\usepackage[unicode]{hyperref}
\usepackage{latexsym}
\usepackage{layout}
\usepackage[english]{babel}
\usepackage{color}
\usepackage{subfigure}
\usepackage{fancyvrb}
\usepackage{color}

\numberwithin{equation}{section}
\theoremstyle{plain}
\newtheorem{theorem}{Theorem}[section]

\newtheorem{Proposition}[theorem]{Proposition}

\newtheorem{conjecture}[theorem]{Conjecture}
\newtheorem{definition}[theorem]{Definition}
\theoremstyle{remark}
\newtheorem{remark}[theorem]{Remark}

\newcommand{\base} {M}
\newcommand{\basehs} {\base_{h,s}}
\newcommand{\basepoint} {m_0}
\newcommand{\basepointone} {\hat\basepoint^1}
\newcommand{\basepointtwo} {\hat\basepoint^2}
\newcommand{\basepointthree} {\hat\basepoint^3}
\newcommand{\boundaryReifenberg} {\edges}
\newcommand{\HHH} {\mathcal H}
\newcommand{\covering} {Y}
\newcommand{\coveringabstract} {Y_\mathrm{abs}}
\newcommand{\ucovering} {\widehat{\covering}}
\newcommand{\ccovering} {\overline{\covering}}
\newcommand{\umin} {u_\mathrm{min}}
\newcommand{\newgroup} {K} 
\newcommand{\particularsurface} {\Sigma}
\newcommand{\particularsurfaceSteiner} {\Gamma}
\newcommand{\genericsurface} {\Sigma}
\newcommand{\genericsurfaceSteiner} {\Gamma}
\newcommand{\pione} {\pi_1(\base)}
\newcommand{\superficieconica} {\Sigma_{\rm c}}
\newcommand{\superficieskew} {\Sigma_{\rm skew}}
\newcommand{\superficiesteiner} {\Sigma_{\rm steiner}}
\newcommand{\superficieminima} {\Sigma_\mathrm{min}}
\newcommand{\superficiminime} {{\mathcal S}_\mathrm{min}}
\newcommand{\surfaceReifenberg} {K}
\newcommand{\psurfaceReifenberg} {V}
\newcommand{\surfacewetting} {\mathfrak{S}}

\newcommand{\edges} {S}
\newcommand{\looops} {C}
\newcommand{\loopone} {\looops_1}
\newcommand{\looptwo} {\looops_2}
\newcommand{\iwires} {\looops_{12}}
\newcommand{\shorts} {S}
\newcommand{\shortone} {\shorts_1}
\newcommand{\shorttwo} {\shorts_2}
\newcommand{\longs} {L}
\newcommand{\longone} {\longs_1}
\newcommand{\longtwo} {\longs_2}
\newcommand{\longthree} {\longs_3}
\newcommand{\longfour} {\longs_4}

\newcommand{\propertyone}{(P1)}
\newcommand{\propertytwo}{(P2)}

\newcommand{\domain} {D}
\newcommand{\domainzF} {\domain_0(\FFF)}
\newcommand{\domainF} {\domain(\FFF)}
\newcommand{\domainzFb} {\domain^{\base}_0(\FFF)}
\newcommand{\domainFb} {\domain^{\base}(\FFF)}
\newcommand{\DD} {\mathbb D}
\newcommand{\FFF} {\mathcal F}
\newcommand{\R} {\mathbb R}
\newcommand{\Sbb} {\mathbb S}
\newcommand{\SSS} {\mathcal S}
\newcommand{\prj} {p}
\newcommand{\prjabstract} {p_\mathrm{abs}}
\newcommand{\minF} {\FFF_\mathrm{min}}
\newcommand{\minFb} {\minF(\base)}
\newcommand{\minFbhs} {\minF(\basehs)}
\newcommand{\projectionofjumpofu} {\prj(J_u)}
\newcommand{\tetrahedron}{\Wedge}
\newcommand{\Wedge}{W}
\newcommand{\Rectangle} {R}
\newcommand{\wiredist} {\rho_\infty}
\newcommand{\wiredistM} {\wiredist(\base)}
\newcommand{\word} {w}
\newcommand{\Z} {\mathbb Z}

\definecolor{mygray}{rgb}{0.92,0.92,0.92}

\newif\ifdraft
\draftfalse

\title{Triple covers and a non-simply connected surface spanning
an elongated tetrahedron and beating the cone}
\renewcommand\shorttitle{A covering space for a tetrahedron}

\author{Giovanni Bellettini}
\address{Dipartimento di Ingegneria dell'Informazione e Scienze Matematiche, Universit\`a di Siena, 53100 Siena, Italy,
and International Centre for Theoretical Physics ICTP, 
Mathematics Section, 34151 Trieste, Italy
}
\email{bellettini@diism.unisi.it}
\author{Maurizio Paolini}
\address{Dipartimento di Matematica e Fisica, Universit\`a Cattolica del Sacro Cuore, 25121 Brescia, Italy}
\email{maurizio.paolini@unicatt.it}
\author{Franco Pasquarelli}
\address{Dipartimento di Matematica e Fisica, Universit\`a Cattolica del Sacro Cuore, 25121 Brescia, Italy}
\email{franco.pasquarelli@unicatt.it}

\date{\today}

\begin{document}

\thanks{}

\begin{abstract}
By using a suitable triple cover we show how to  possibly  model the construction of
a minimal  surface with positive genus  spanning all six edges of a  tetrahedron, 
working in the space of 
BV functions and interpreting the film as the  boundary of a Caccioppoli set
in the covering space.
After a question raised by R. Hardt in the late 1980's, 
it seems  common opinion that an area-minimizing surface of this sort does not exist for a regular
tetrahedron, although a proof of this fact is still missing.
In this paper we  show that there exists a surface of positive genus spanning the boundary of an elongated
tetrahedron and having area strictly less than the area of the conic surface.
\end{abstract}

\maketitle


\section{Introduction}\label{sec:intro}
Finding a soap film that spans all six edges of a regular tetrahedron different from
the cone  of Figure \ref{fig:morgan} (left) is an
intriguing problem.
It was discussed by Lawlor and Morgan in \cite[p. 57, and fig. 1.1.1]{LaMo:94}, where a sketch of a possible
soap film of positive genus is shown,
 based on an idea of R. Hardt\footnote{
A sketch of this surface was reportedly found
by F. Morgan in R. Hardt's office during a visit at Stanford
around 1988.
F. Morgan and J. Taylor tried to find crude estimates to compare the two minimizers without
success.
Many years later R. Hardt himself told R. Huff about the problem, who then came out with the
results that can be found in \cite{Hu:11}.
}; such a surface is 
here reproduced in Figure \ref{fig:morgan} (right)%
\footnote{
The picture itself is a computer generated image obtained by J. Taylor using the 
\texttt{surface evolver} of K. Brakke.
}.
The same picture was subsequently included in the book \cite[fig. 11.3.2]{Mo:08}.

The cone constructed from the center of the solid spanning the six edges 
of the regular tetrahedron (Figure \ref{fig:morgan} left)
has been
proved to be area-minimizing\footnote{
That is, $({\bf M}, 0, \delta)$-minimal
in the sense of F.J. Almgren \cite{Al:76}.} if,
roughly speaking, one imposes on the competitors the extra constraint that they 
divide the
regular tetrahedron in four regions, one per face \cite[Theorem IV.6]{Ta:76}, see also \cite{LaMo:94}.
It corresponds to the actual shape that a real soap film attains when dipping a
tetrahedral frame in soapy water;
it includes a $T$-singularity at the center, where four triple lines ($Y$-singularities) converge
from the four vertices satisfying the local constraints of an area-minimizing surface \cite{Ta:76}.

However, it is an open question whether a non-simply connected film,
e.g. with ``tunnels''
 connecting pairs of faces, could beat the cone;
 Figure \ref{fig:morgan}
(right)
 shows a theoretically feasible configuration of such a minimizing film.

Although it seems a common opinion, based both on physical experiments and theoretical reasons, that
such a surface does not actually exist (see e.g. \cite{Hu:11}), to our best knowledge such a
question still remains open.

Generalizations of the problem, for instance considering deformations of the
tetrahedron with the addition of zig-zags \cite{Hu:11}, allows on the contrary to construct
a surface with the required topology that at least satisfies all local properties of a minimal
film.
Another generalization that could lead to interesting minimizers consists in considering
anistotropic surface energies \cite{LaMo:94}.

The surface with positive genus\footnote{
This surface contains triple curves and boundaries,
in this context we could not find a recognized definition of genus.
However, if we remove the two flat portions, each bounded by a side of the tetrahedron
and one of the two triple curves (operation that can be performed by deformation retraction, hence without
changing the topological type), we end up with a surface without triple lines and bounded by a skew
quadrilateral.
We can then apply the formula $\chi = 2 - 2g - b$ where $\chi$ is the Euler characteristic
($\chi = -1$ in our case),
$g$ is the genus and $b = 1$ is the number of components of the boundary, obtaining $g=1$.
Consistently, the retracted surface can be deformed into a disk with a handle, or equivalently
into a puntured torus.}
depicted in Figure \ref{fig:morgan} includes two triple curves
(curves where three sheets of the surface meet at $120^\circ$) and no quadruple points.
Furthermore it sports two tunnels, one clearly visible that
allows to traverse the  tetrahedron entering from the front face and exiting from the back.
The other hole is located on the other side of the film and allows to traverse the tetrahedron
entering from the right lateral face  and exiting from  the bottom face (without crossing
the soap film). Figure \ref{fig:surfacefilm} in Section \ref{sec:comparison} helps to figure out
the topological structure.

\begin{figure}
\includegraphics[width=0.55\textwidth]{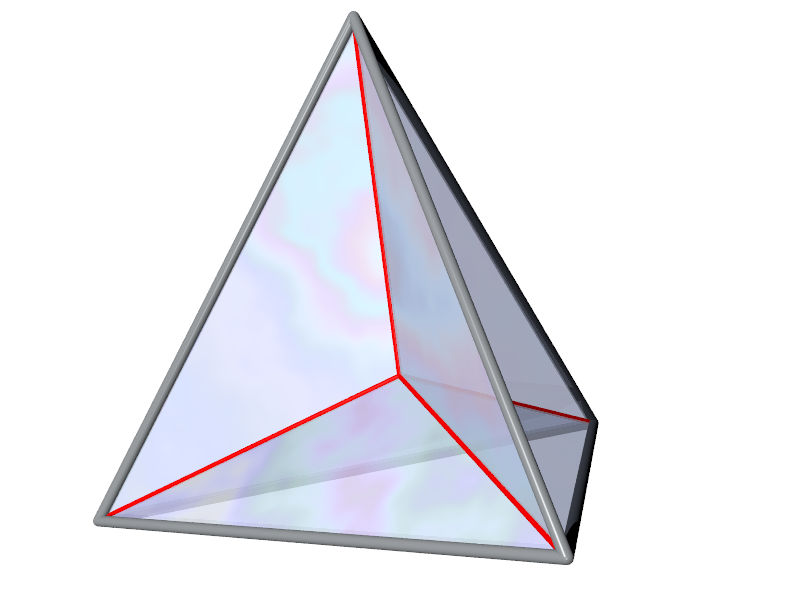}
\includegraphics[width=0.43\textwidth]{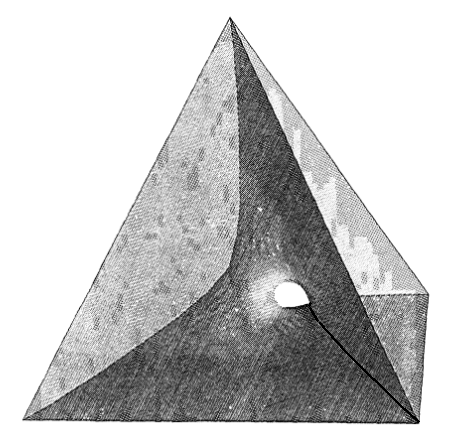}
\caption{\small{Left: the classic conical minimal film 
spanning a regular tetrahedron.
Right: a slightly retouched version of \cite[fig. 1.1.1]{LaMo:94}.
In turn it is an enhanced version of \cite[fig. 11.3.2]{Mo:08}.
Each triple curve ``passes through'' one of the two tunnels. 
}}
\label{fig:morgan}
\end{figure}

Our first result (Section \ref{sec:covering}, on the basis also of 
the computation of the fundamental group in Section \ref{sec:fundamental}) 
is the construction of a covering space of degree $3$ 
of the complement of the one-skeleton of a tetrahedron,
following the lines of  
\cite{Br:95} (see also \cite{AmBePa:15}),
that is compatible with Figure \ref{fig:morgan},  right.
Using covering spaces allows to treat in a neat way situations that seem hard to
model using other approaches; see for instance
Section \ref{sec:reifenberg}, where we compare our approach
with the Reifenberg approach.

A small portion of the soap film somehow behaves like a sort of ``portal'' to a
parallel (liquid) universe.
More precisely, each point in $\R^3$, after removal of a suitable set of curves
(obtaining the so-called base space)
has two other counterparts, for a total of three copies of the base space
that are actually to be interpreted (locally) as three distinct sheets of a
cover.
Globally the picture is more interesting, since the covering space is constructed
in such a way that when travelling along a closed curve in the base space, the
``lifted'' point might find itself on a different sheet of the same fiber.
This can be used as a trick to overcome the problem in treating the soap film
as transition between air and liquid.
Since the liquid part has infinitesimal thickness, this
would lead to the superposition of two layers, one corresponding to the air-liquid transition and the other to the
transition back from liquid to air.
For this reason an approach based e.g. on $BV$ functions or Caccioppoli sets would lead to a
liquid phase of measure zero and miss completely the two
superposed layers.
Using the covering space overcomes the problem by adding a ``fake'' big set of
liquid fase, lying in a different sheet of the same fiber with an entry point
corresponding to one face of the soap film and an exit point (reached travelling
along suitable closed curves in the base space) from the other face 
in the same position.
A phase parameter $u$ defined in the covering space is then defined with values in
$\{0,1\}$, $0$ indicating liquid, $1$ indicating air, in such a
way that in exactly one of the three points of a fiber we have $u = 1$ (air).
Looping around an edge of the tetrahedron would take from one point to another
of the same fiber, thus forcing the value of $u$ to jump somewhere along the
loop, which in turn would force the soap film to ``wet'' the edge.

The presence of triple curves implies that at least three sheets are required for a
cover modelling the film problem, however the natural construction
using suitable cyclic permutations of the three sheets when circling around each of the
six edges of the tetrahedron cannot produce a surface with holes.
This is because any path that traverses a tetrahedron entering from any face and exiting
through another one is by topological reasons linked to exactly one of the edges and hence
forced to traverse the film.
Some way to distinguish tight loops around an edge and loops that encircle the edge far away
is then required.

Hence the construction is more involved and requires the introduction of two
``invisible wires''.
This is done in the same manner and for similar reasons as in \cite[Section 5.1]{Br:95},
see in particular examples 7.7 and 7.8 in that paper. After constructing the covering space $\covering$,
we can adapt the machinery of \cite{AmBePa:15} (see Section \ref{sec:functional}), and settle the Plateau problem
in $BV$, with the differences that here the 
cut surfaces have selfintersections,
and  that the involved functions defined on $\covering$, instead of taking
values in an equilateral triangle (with barycenter at the origin) with the constraint of having zero sum 
on each fiber,
take here values in $\{0,1\}$,
with the restriction  indicated
in \eqref{eq:domain} and discussed above.

Our next main result (Theorem \ref{teo:comparison}) is to prove that, for a sufficiently elongated tetrahedron, there is a surface spanning
its boundary, having the topology of the surface of Figure \ref{fig:morgan} right, and 
having area {\it strictly less} than the area of the conelike configuration. 
Therefore, if we allow for competitors of higher genus, 
we expect the conelike surface  not to have minimal area.
We remark that our result does not cover
the case of a {\it regular} tetrahedron.

Positioning the invisible wires is delicate. Indeed, 
we would like the invisible wires not to influence the 
minimal film, which requires that
the film does not wet them.
This is not proved here, although the numerical simulations strongly support this fact for
a sufficiently elongated tetrahedron and suitably positioned invisible wires.
On the other hand, 
the discussion in Section \ref{sec:whatcangowrong} shows that a nonwetting
relative $BV$-minimizer does not exclude the existence of an absolute minimizer with the
structure of Figure \ref{fig:surfacefilmwrong}, which we would like to exclude.
Again, the numerical simulations support the conjecture that if the invisible wires are
positioned sufficiently far away from the short edges, then the absolute minimizer has the
required topology and does not wet the invisible wires.
On the contrary, positioning the invisible wires near the short edges would produce
absolute minimizers that partially wet the invisible wires.
We observe that soap films that partially wet a given curve were discussed in \cite{Al:76}
and proved to exist for any knotted curve in \cite{Pa:92}.
In Section \ref{sec:allcoverings} we describe all possible triple covers of the base space
among those producing soap films wetting all the edges of a tetrahedron. We 
conclude the paper with Section \ref{sec:conclusions}, where we describe
the results of some numerical simulations, in particular varying the length
of the long edges of a tetrahedron.


\section{The base space}\label{sec:base}
In view of the symmetry%
\footnote{The symmetry group of the surface of Figure \ref{fig:morgan}
(right) turns out to be $\DD_{\textrm{2d}}$, using Schoenflies notation.
}
of the desired surface, it is convenient to
think of the tetrahedron as a wedge with two short edges, $\shortone$ and $\shorttwo$, and four long edges,
$\longs_i$, $i = 1, \dots, 4$; for simplicity we name the one-skeleton 
of the wedge, i.e. the union of all edges, as
$\edges = (\cup_{i=1}^2 \shorts_i) \cup (\cup_{i=1}^4 \longs_i)$.
A sketch of such a wedge is displayed in Figure \ref{fig:wedge}.

\begin{figure}
\includegraphics[width=0.5\textwidth]{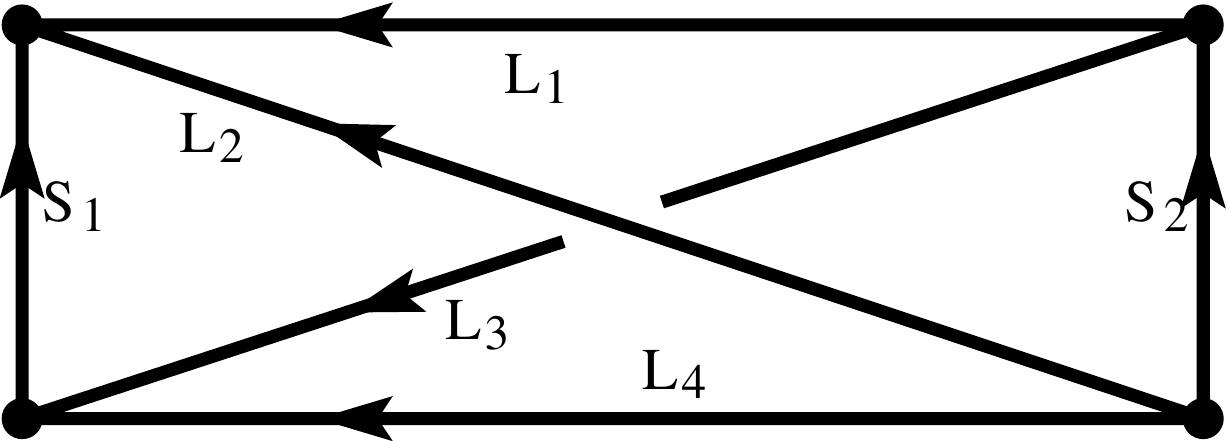}
\caption{\small{In view of the symmetry $\DD_{\textrm{2d}}$ of the desired surface, it is convenient to
think of the tetrahedron as a wedge with two short edges, $S_1$ and $S_2$, and four long edges,
$L_1, L_2, L_3, L_4$.
When required, e.g. in \eqref{eq:arcs},
 the short edges are oriented upwards, and the long edges are oriented from right
to left.}}
\label{fig:wedge}
\end{figure}

In order to construct a proper base space for our cover  let us summarize the required properties of the
soap film that we would like to obtain.

\begin{enumerate}
\item{}
The soap film is required to wet all six edges of the tetrahedron/wedge.  From the point of view of a
covering space this is achieved by requiring that any closed
path that circles once around an edge at
a short distance acts on the fiber with no fixed points;
\item{}
it is possible to find closed paths that suitably traverse the tetrahedron (with reference to
Figure \ref{fig:morgan}, one path entering from the front face
and exiting from the back face, the other entering from the right face and exiting from the bottom face)
that when lifted to the covering space do not move at least one point of a fiber.
\end{enumerate}

The second requirement seems at first sight incompatible with the first
one, since such traversing paths are
actually forced to circle once around a single edge of the tetrahedron.  This is unavoidable and a
consequence of the topology of the graph having the six edges as arcs.
However these paths  are allowed to circle the edges ``far away'', and we can make the two paths, the one
circling (say) the edge on the left ($\shortone$ in Figure \ref{fig:wedge}) at a
short  distance and the one
traversing the visible hole in the soap film in Figure \ref{fig:morgan}, not homotopically
equivalent by introducing an obstruction in the base space
in the form of an invisible wire, displayed as the left circle in Figure \ref{fig:covering_base}.
The invisible wire has the purpose of making the two paths not equivalent, but in the meantime we do not want it
to ``perturb'' the soap film.

\begin{figure}
\includegraphics[width=0.7\textwidth]{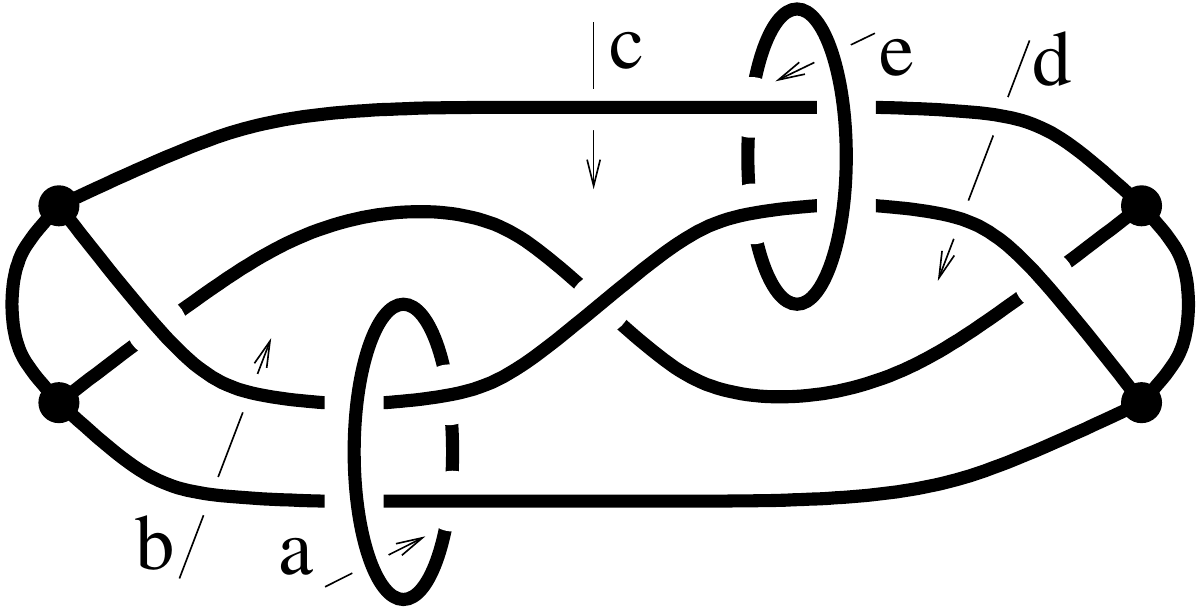}
\caption{\small{The base space $\base$  is $\R^3$ with the displayed arcs removed.
Arrows and labels correspond to the generators of a presentation of
the fundamental group.}}
\label{fig:covering_base}
\end{figure}

We actually need to introduce two invisible wires, in the form of two closed loops $\loopone$, $\looptwo$ suitably interlaced
to the edges of the wedge.
We name 
their union as $\iwires = \loopone \cup \looptwo$.
The result is illustrated in Figure \ref{fig:covering_base}, the 
base space 
$$
\base = \R^3 \setminus (\edges \cup \iwires)
$$
being the complement
in $\R^3$ of the system of curves displayed as thick lines.
We shall use it to construct the covering space $\covering$.

The picture follows the usual convention of inserting small gaps to denote
\emph{underpasses} of a curve below another.
The four dots (vertices of the tetrahedron) represent points where three curves meet.
This system of curves is the disjoint union of two loops ($\loopone$ and $\looptwo$),  a set of
two short curves ($\shortone$, $\shorttwo$) and four long curves ($\longone$ to $\longfour$) joining four
points.
The latter is topologically equivalent to the set of edges of a tetrahedron.

Each of the two $\loopone$ and $\looptwo$ loops around a pair of long edges, they are called \emph{invisible
wires} in \cite{Br:95} and their presence is essential to allow for jump sets 
(see Section \ref{sec:functional}) with the required topology.
The loop $\loopone$ is the one nearest to the short edge $\shortone$.

The quantity in the next definition plays an important role.

\begin{definition}
Given a choice of the geometry of $\base$, $\wiredistM$ is defined as the $L^\infty$-distance\footnote{%
The $L^\infty$-norm $\|x\|_\infty := \max(|x_1|,|x_2|,|x_3|)$ is used here for convenience in view of the
estimates to follow.}
 from the two
short edges and the invisible wires of $\base$:
$$
\wiredistM := \min \{ \Vert x - \xi \Vert_\infty : x \in \shortone \cup \shorttwo, \xi \in \loopone \cup \looptwo \} = {\rm dist}_{L^\infty}(S_1 \cup S_2, 
C_1 \cup C_2).
$$
\end{definition}

\begin{remark}
In constructing the base space $\base$ we did not pay much emphasis on its geometry (see e.g.
Figure \ref{fig:covering_base}),
which is allowed  as long as we study topological properties like its fundamental group, or when constructing
the covering space $\covering$.%
\footnote{It should be noted here that 
the base space $\base$ is
path connected, locally path connected and semilocally simply-connected \cite[Chapter 1.3]{Ha:02}.}
However, when considering the minimal film, the actual geometry becomes important.
We shall then make specific choices both for the set of curves corresponding to the tetrahedral
frame, straight segments with two different lengths, and for the two closed curves corresponding to
the invisible wires.
We point out here that the two invisible wires can be safely deformed into straight lines, 
one in the $z$ direction through $(-s,0,0)$, 
the other in the $y$ direction through $(s,0,0)$, 
for a suitable
choice of $s>0$, observing that a straight line is a closed curve in the compactification
$\Sbb^3$ of $\R^3$.
To avoid problems at infinity, where the two invisible wires would intersect, we can deform one or
both of them ``far away'' (outside the convex hull of $\edges$).
\end{remark}

\begin{definition}\label{def:geometry}
For two parameters  $h > 0$ and $s \in (0, 1)$, we define a specific geometry for $\base = \basehs$ as follows.
The four vertices of the wedge $\tetrahedron$ are fixed at
\begin{equation}\label{eq:coordinates_vertices}
(h, \pm 1, 0), \qquad (-h, 0, \pm 1),
\end{equation}
and connected with straight segments of length $|\shortone| = |\shorttwo| = 2$,
$|\longs_i| = \sqrt{2 + 4h^2}$, $i = 1,\dots,4$.
The invisible wires are now selected as straight lines 
$\loopone = \{(-s h,0,t): t \in \R\}$
and
$\looptwo =  \{(s h,t,0): t \in \R\}$,
(possibly modified far away from $\tetrahedron$).
\end{definition}

The special value $h = \frac{\sqrt{2}}{2}$ results in a regular tetrahedron, whereas for $h > \frac{\sqrt{2}}{2}$
we have $|\longone| > |\shortone|$.

Clearly, we have $\wiredist(\basehs) = h(1-s)$.


\section{Computing the fundamental group}\label{sec:fundamental}
We shall occasionally need to fix a base point $\basepoint$ in $\base$.
It is positioned far away from the set of curves, its actual position is
inessential and we shall think of it as the eye of the observer.
Equivalently, we can position it at infinity after compactification of
$\R^3$ into $\Sbb^3$.%
\footnote{Since a single point into a three-space cannot obstruct a closed path,
adding the point at infinity to $\R^3$ does not impact the computation of the
fundamental group, nor it will make any difference in the construction of the
covering space.
For that matter, it also makes no difference to substitute $\R^3$ with a big
ball compactly containing the tetrahedron.}

The fundamental group $\pione$ can be computed by using a technique similar
to the construction of the Wirtinger presentation of a knot group.
We position the base point $\basepoint$ above the picture 
and select
a set of closed curves $a$, $b$, $c$, $d$, $e$ that will serve as generators
of the group.
These curves (that represent elements of $\pione$) are displayed in Figure
\ref{fig:covering_base} as arrows to be interpreted as curves that start at $\basepoint$, 
run straight to the tail of one of the arrows, follow the arrow
\emph{below} one or two arcs of the system of curves and finally go back straight
to $\basepoint$.

In order to prove that $a$, $b$, $c$, $d$, $e$ generate the whole fundamental group it
is enough to construct curves
that loop around each of the pieces of curves
running from an underpass or node to another underpass or node.

As an example, the product $c^{-1} d$ is equivalent to a curve that loops around
the piece of intermediate edge running from one disk to the other
($L_{2,2}$ in the notation fixed below).
This can be seen by observing that loop $d$ can be dragged from the right
to the left of the top-right circle.

At this point we can construct $b c^{-1} d$ looping around the bottom-right
piece of curve $L_{4,2}$ from the underpass to the lower-right node.

Curve $c^{-1} b$ corresponds
to one of the four arcs 
($L_{3,2}$ in the notation below) in which the long edge
connecting the lower-left node  to the upper-right one is divided.
This can be seen by observing that modifying $b$ by extending its head to pass under
$L_{3,2}$ (Figure \ref{fig:covering_base}) gives a curve that is homotopic to $c$.

Traversing an underpass can be achieved by conjugation with the loop corresponding
to the overpass, which allows to obtain all curves associated to the long
edges.
Product of two of these finally allows to loop around the short edges on the left
and on the right.
We end up with the following table, where the second index denotes what piece of the
long arc of the wedge we are referring to (from left to right):
\begin{equation}\label{eq:arcs}
\begin{aligned}
& L_{1,1}\to  c, \qquad L_{1,2}\to e c e^{-1} ,
\\
& L_{2,1}\to  a c^{-1} d a^{-1}, \qquad L_{2,2}\to c^{-1} d, \qquad L_{2,3}\to e c^{-1} d e^{-1} ,
\\
& L_{3,2}\to c^{-1} b, \qquad L_{3,3} \to d^{-1} b c^{-1} d ,
\\
& L_{4,1}\to  a d^{-1} c b^{-1} a^{-1} , \qquad L_{4,2} \to d^{-1} c b^{-1} ,
\\
& S_1 \to a d^{-1} c a^{-1} c^{-1}, \qquad S_2 \to b c^{-1} e c e^{-1} .
\end{aligned}
\end{equation}
We omit the values associated to $L_{3,1}$ and $L_{3,4}$ (readily deducible by conjugation due to
traversal of, respectively, $L_{2,1}$ and $L_{2,3}$), since we shall not need them.

Each crossing provides a relation among the three curves involved.  By collecting
all such relations and simplifying
we finally end up with the presentation
(five generators and two relators)%
\footnote{This is actually a right-angled Artin group.}
\begin{equation}\label{eq:presentation}
\pione = <a, b, c, d, e; ab = ba, de = ed>.
\end{equation}

A different and more direct way to obtain the presentation \eqref{eq:presentation} consists
in an ambient deformation of a tubular neighborhood 
of the set of curves of Figure
\ref{fig:covering_base}.  
An important remark here is that it is possible to flip the
configuration of a ``Steiner-like'' pair of adjacent triple junctions as shown in Figure
\ref{fig:steinerflip},
 without changing the homotopy type of both the set of curves and of
its complement in $\R^3$.
This allows to transform the set of curves of Figure \ref{fig:covering_base} by homotopy
equivalence into the first configuration of Figure \ref{fig:bouquet}.
Then we shrink two curves to a point in the passage from the second to the third configuration, 
and again one curve from the third to the last without modifications to the homotopy type of
both the set of curves and of its complement in $\R^3$.

\begin{figure}
\includegraphics[width=0.9\textwidth]{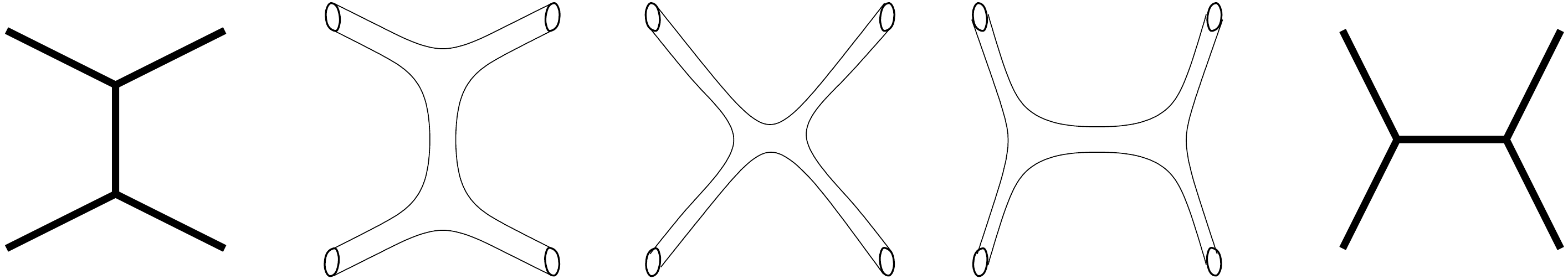}
\caption{\small{A tubular neighborhood of the Steiner tree on the left can be
ambiently deformed into a tubular neighborhood of the Steiner tree on the right.}}
\label{fig:steinerflip}
\end{figure}

Using this equivalence we can deform the set of curves as shown in Figure \ref{fig:bouquet}
to a bouquet of three loops with two linked rings, a configuration consistent with
the presentation \eqref{eq:presentation} for the fundamental group of the complement.

\begin{figure}
\includegraphics[width=0.25\textwidth]{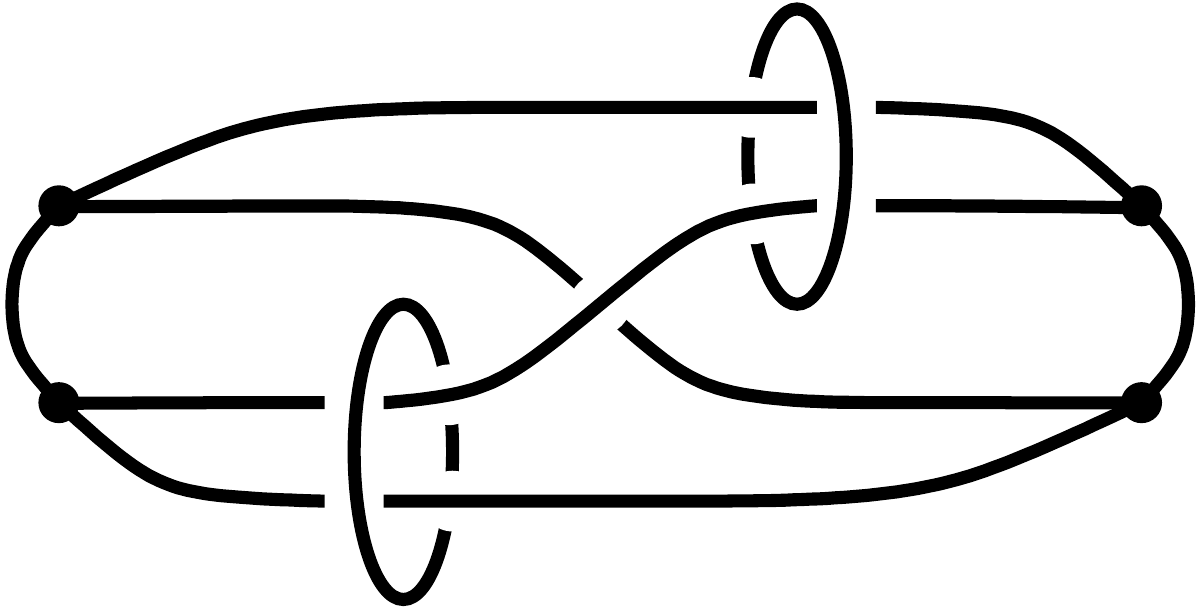}
\includegraphics[width=0.25\textwidth]{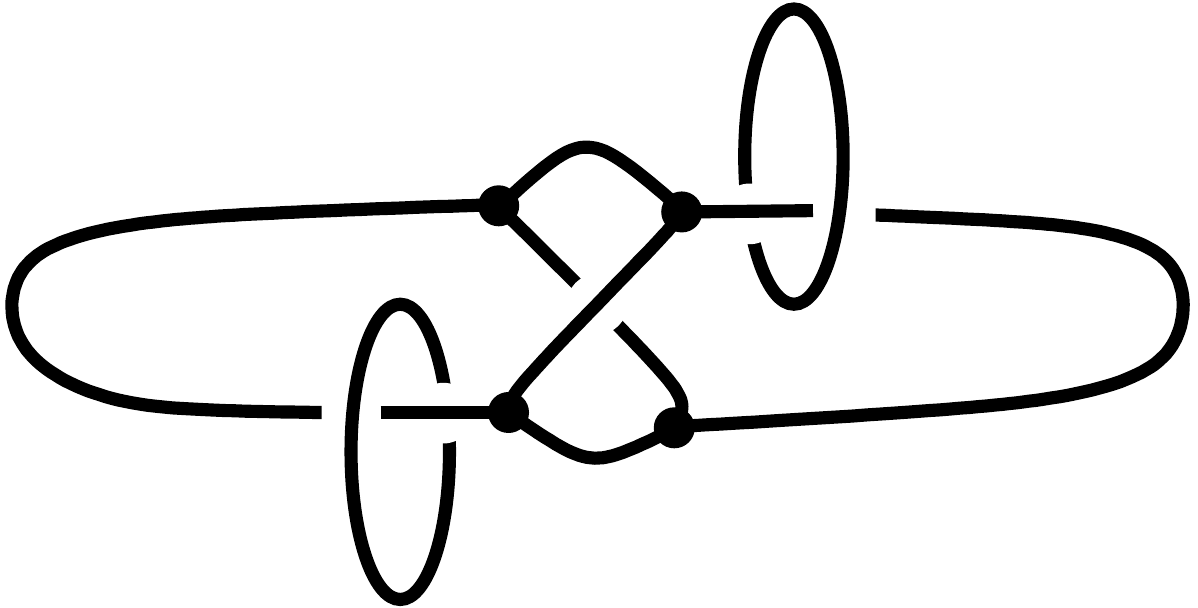}
\includegraphics[width=0.25\textwidth]{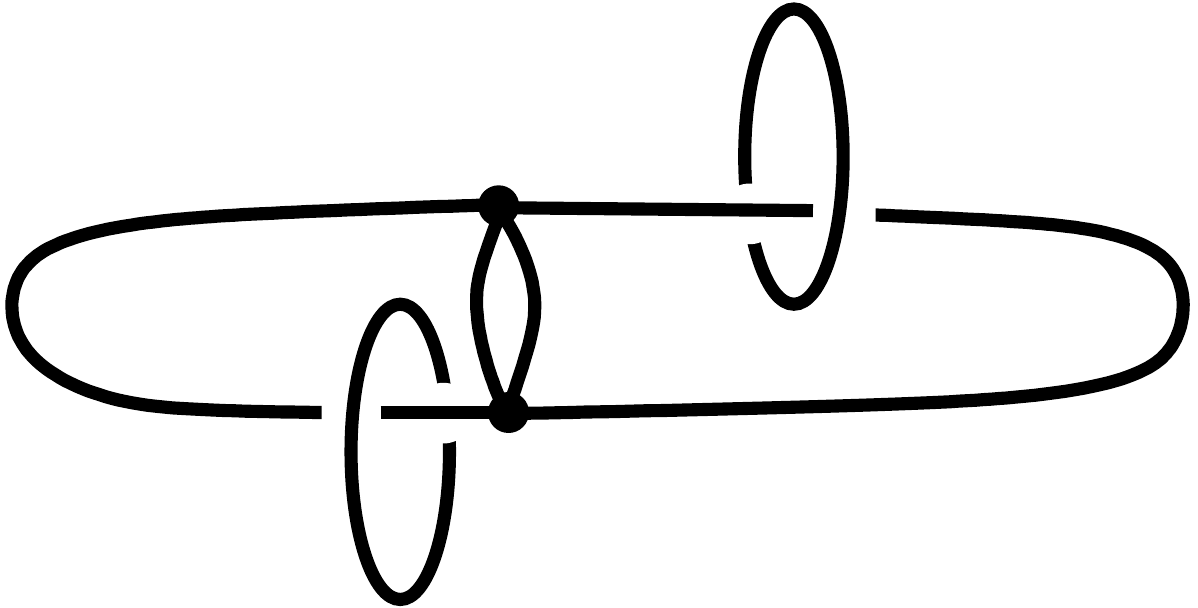}
\includegraphics[width=0.15\textwidth]{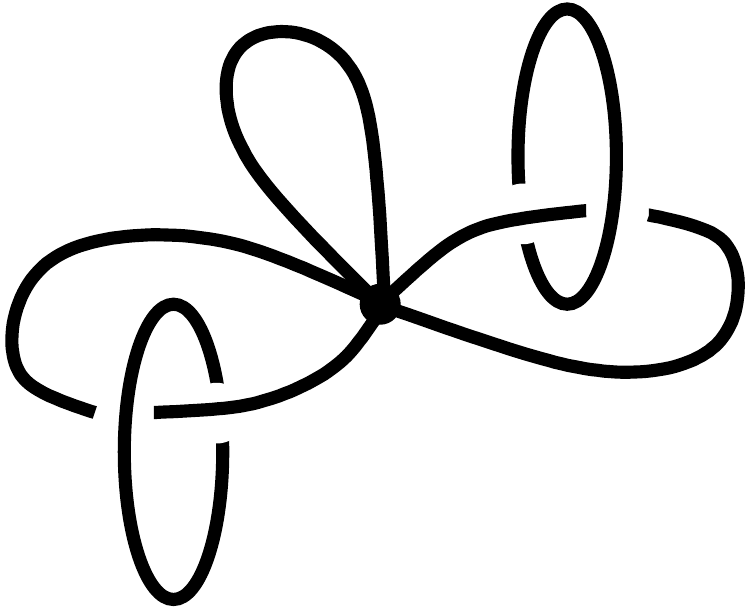}
\caption{\small{A sequence of steps showing that a tubular neighborhood of the system of curves
can be ambiently deformed
into a (tubular neighborhood of a) bouquet of three loops with two linked rings.}}
\label{fig:bouquet}
\end{figure}

It should be noted that the graphs in the sequence of Figure \ref{fig:bouquet} are not mutually homeomorphic,
nor they are homeomorphic to the system of curves of Figure \ref{fig:covering_base}, whereas their complement
in $\R^3$ is diffeomorphic.
We have an ambient isotopy as soon as we substitute each system of curves with a small tubular neighbourhood.

\section{The triple cover}\label{sec:covering}
The presence of triple curves in the soap film that we would like to reconstruct (Figure \ref{fig:morgan}, right) implies that 
the covering space must be at least of degree three.
There is no quadruple (tetrahedral) point in this film, so that three sheets for the cover
might be sufficient.

However, the particular structure of the surface (a single smooth component of the film that
arrives to a triple line from two distinct directions) requires special treatment, similar to
that of Examples 7.7 and 7.8 of \cite {Br:95}, with the introduction of the so-called
\emph{invisible wires}.
These are introduced as two copies of $\mathbb S^1$ having 
the purpose of exchanging sheets $2$ and
$3$, whereas sheet $1$ is glued to itself.

As we shall see in Section \ref{sec:allcoverings} the introduction of the invisible wires
is essential, in the sense that a cover with three (or two) sheets of the complement of
the six tetrahedral edges is incompatible with the wetting conditions imposed on all edges.

A \emph{cut and paste} construction of our cover is as follows \cite{AmBePa:15}.

\begin{figure}
\includegraphics[width=0.7\textwidth]{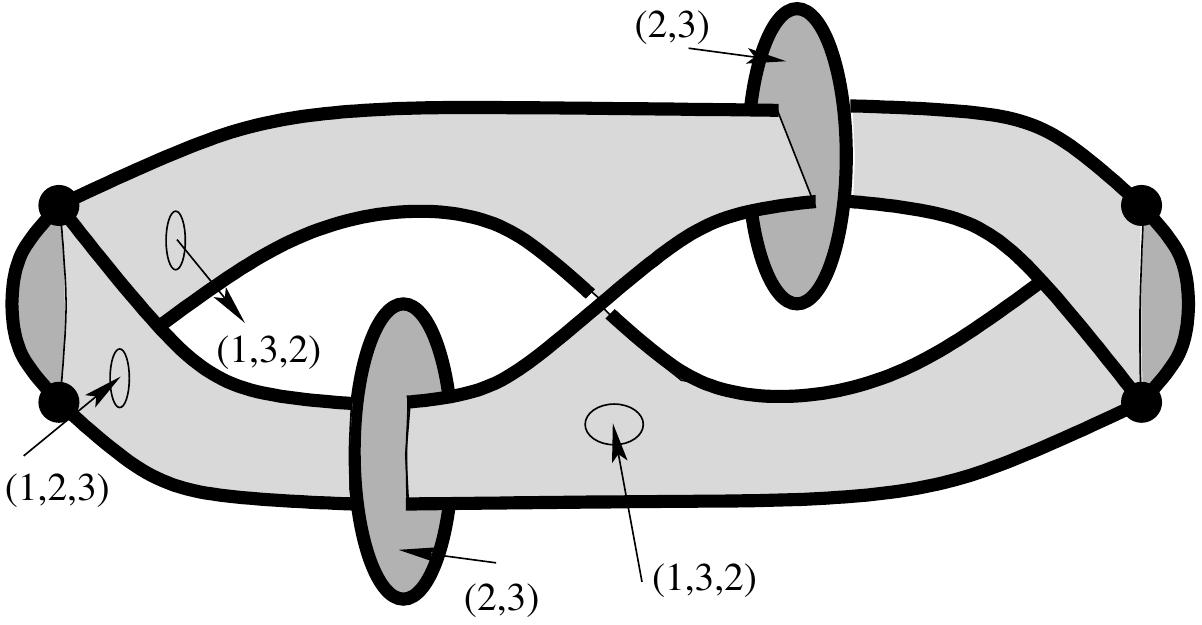}
\caption{\small{The cover is defined by a cut and paste technique on three copies of the
base space cutted along the shaded surfaces. Local orientation is indicated, 
as well as the gluing among the sheets, with cycle notation.
}}
\label{fig:covering_cut}
\end{figure}

\begin{itemize}
\item
Take three copies of the base space $\base$ (complement in $\R^3$ of the system of curves shown in
Figure \ref{fig:covering_base}): sheet $1$, $2$ and $3$;
\item
cut them along the shaded surfaces\footnote{%
The set of shaded surfaces and the system of curves give rise to a stratification, where the
two-dimensional stratum (the cutting surfaces) is divided into connected components.
The gluing permutation locally associated to each component can be transported along the whole
component and must close consistently along closed paths on the surface.
This would be a problem for non-orientable components (which is not our case) unless the permutation
is of order two.
In particular this is not an issue for covers of degree $2$.
}
displayed in Figure \ref{fig:covering_cut}; 
\item glue the three sheets again in such a way that when crossing the large surface cut the three sheets
are glued cyclically.
When crossing the shaded disks sheets $2$ and $3$ get exchanged whereas sheet $1$ on one side is glued to
sheet $1$ on the other side.
\end{itemize}

\begin{remark}[Cutting locus as a stratified set]\label{rem:stratified_cutting}
The set of cutting surfaces of Figure \ref{fig:covering_cut} (the cutting set) forms a stratification composed by seven
pieces of (2D) smooth surfaces joined by four (1D) arcs (thin arcs in the picture, two triple curves
and two selfintersection curves) and four (0D) end-points of the two selfintersection curves.
The stratification as a whole is bounded by the set of curves defining $\base$ (thick curves in the
picture) and by the four vertices of $\tetrahedron$.
Each piece of 2D surface is orientable.
The arrows indicate a choice of orientation,
the two small lunette-like dark regions on the left and on the right are oriented from front to back
(however they are not a critic part of the cutting locus, see  Remark \ref{rem:lunette})
and finally the small portion of surface on the right of the upper disk is oriented from back to front,
consistent with the orientation on the left of the disk.
It should be noted that the large piece of surface connecting the two disks is subjected to a
twist in the region near the center of the picture, so that the top portion is oriented from back to
front.
The gluing is based on the permutations shown in the picture
(expressed in cycle notation) when crossing the surface in the direction of the arrows.
Local triviality entails a constraint at
the two intersection curves between the disks and the other surfaces:
a tight loop around one of such intersections is contractible in $\base$, hence the composition of
the four permutations associated to crossings of the cuts (or their inverse if the
loop crosses in the opposite direction with respect to the cut orientation) must give the identity.
This results in a constraint upon the
permutations on the left and on the right: they must be related by a conjugation defined by
the transposition $(2,3)$ associated to the disks.
On the lunette surface to the left the permutation is $(1,3,2)$ when crossing from front to back.
It cannot be chosen arbitrarily because of the local triviality property:
a tight loop around the triple curve is contractible, hence the product of the three permutations associated to crossings of the cuts
(or their inverse, according to orientation) must be the identity.
Similarly, the permutation associated to the right lunette is forcibly given by $(1,2,3)$.
\end{remark}

\noindent We denote by $\prj: \covering \to \base$ the cover defined in this way.

\begin{remark}\label{rem:lunette}
It is not actually necessary to use triple curves in the definition of the cover,
indeed the left and right small fins could be removed and the two sides of the large surface could be
extended up to the short lateral edges of the tetrahedral wedge.
The chosen cut surface just mimics the structure that we expect for the minimizing film.
\end{remark}

\begin{remark}\label{rem:metric}
The local triviality of the cover allows to naturally locally endow the covering space $\covering$
with the euclidean metric induced by $\prj$ from the euclidean structure of $\base$.
\end{remark}

\begin{remark}\label{rem:cut_deformation}
The abstract construction in particular implies that, up to isomorphisms, the covering space constructed
by cut and past is independent of the actual geometry of the cutting set, provided it has the same
structure and the same gluing permutations at corresponding points.
More precisely, the covering space is the same (up to isomorphisms) if the cutting set is deformed
using a homeomorphism of $\R^3$ into itself with compact support and fixing the edges of $\Wedge$
and the invisible wires $\iwires$, and the gluing permutations are defined consistently.
\end{remark}

\subsection{The cover is not normal}\label{sec:notnormal}

The cover $\prj: \covering \to \base$ is clearly path connected.
We claim that it is not normal \cite[Chapter 1.3]{Ha:02},
as a consequence of the fact that sheet $1$ is somehow specially treated by the gluing performed
at the two disks.
We recall that a cover is normal if for any pair $y, \eta \in \covering$ with $\prj(y) = \prj(\eta)$ there
exists a deck transformation\footnote{
A deck transformation is a homeomorphism $\psi : \covering \to \covering$ such that
$\prj (\psi(\eta)) = \prj(y)$ for any $y \in \covering$.
}
 $\psi : \covering \to \covering$ with $\psi(y) = \eta$.

\begin{Proposition}\label{prop:the_covering_is_not_normal}
The cover $\prj: \covering \to \base$ is not normal.
\end{Proposition}

\begin{proof}
It is sufficient to show that 
the identity is the only deck transformation.
Suppose by contradition that $\psi$ is  a nontrivial deck transformation. Then
$\psi$ has no fixed points \cite[page 70]{Ha:02}.
Now take $y \in 
 \prj^{-1}(\basepoint)$ in sheet $1$, 
 $\basepoint$ being the base point of $\pione$.
Then $\psi(y)$ belongs to either sheet $2$ or $3$; 
suppose for definiteness that it belongs to sheet $2$.
If $\gamma$ is a closed path in $\base$ corresponding to $a$ (Fig. \ref{fig:covering_base})
of $\pione$, we can lift $\gamma$ to $\covering$ into two
distinct paths, one starting at $y$, the other starting at $\psi(y)$.
These paths are mapped into each other by the homeomorphism $\psi$, however one is closed (the one starting
at $y$), whereas the other is open, since when traversing the disk, the lifted path will continue on sheet $3$.
This gives a contradiction.
\end{proof}


\subsection{Abstract definition of the covering space}\label{sec:abstract}

It is well known that an abstract definition of a covering space
$\prjabstract : \coveringabstract \to \base$
of $\base$ is based on selecting a subgroup $H$ of $\pione$,
considering
 the set $\ucovering$ of paths $\gamma : [0,1] \to \base$ with $\gamma(0) = \basepoint$,
and
taking the quotient with respect to the equivalence relation
\begin{equation}\label{eq:abstractequiv}
\gamma_1 \sim \gamma_2 \iff
\gamma_1(1) = \gamma_2(1) \text { and }
[\gamma_1 \gamma_2^{-1}] \in H ,
\end{equation}
where $\gamma_2^{-1}$ denotes the path $\gamma_2$ 
with opposite orientation, and defining the projection from
$\ucovering$ to $\base$ as $[\gamma] \to \gamma(1)$.
The degree of the cover is given by the index of $H$ in $\pione$.
We shall describe a procedure to produce a subgroup $H$ of index $3$ in $\pione$
(finitely presented in \eqref{eq:presentation}) and subsequently prove in Theorem \ref{teo:abstract}
that it gives a cover isomorphic to $\prj : \covering \to \base$.

In order to construct $H$ 
we need a concrete way to identify its elements when written as
words in the generators of the presentation \eqref{eq:presentation}.
The first task is then to compute the actions
$\sigma_a$,
$\sigma_b$,
$\sigma_c$,
$\sigma_d$,
$\sigma_e$
on the fiber $\{\basepointone,\basepointtwo,\basepointthree\}$ over the base point
$\basepoint \in \base$ (the superscripts refer to the three sheets $1, 2, 3$),
corresponding to each generator in \eqref{eq:presentation}.
This amounts in associating to each generator the resulting permutation induced on sheets
$1$, $2$, $3$.
A quick check comparing Figures \ref{fig:covering_base} and \ref{fig:covering_cut} suggests to define
$$
\sigma_b = \sigma_d = () , \qquad
\sigma_a = \sigma_e = (2,3) , \qquad
\sigma_c = (1,3,2) ,
$$
where $()$ denotes the identity permutation. 
Observe that $\sigma_a$ and $\sigma_b$ commute, as well as $\sigma_d$ with $\sigma_e$, so that
the two relators in \eqref{eq:presentation} are consistent with these actions.
Given an element of $\pione$ expressed as a word $\word$ in the generators, 
by substituting these actions to the generators in $\word$ and performing the
multiplications (left to right), we are then able to compute the action of the element
represented by $\word$ on the fiber $\{\basepointone,\basepointtwo,\basepointthree\}$
in terms of a permutation of the three superscripts.
$H$ will then be recovered as consisting of those words that produce a permutation fixing $1 \in \{1,2,3\}$.
Using relations satisfied by the actions $\sigma_a$ through $\sigma_e$ we can simulate the final
multiplication after substitution in $\word$ by imposing such relations directly on the generators,
the result would be the same.
So we can safely add such relations to the presentation \eqref{eq:presentation} as
extra relators\footnote{The presence of $a^2$, $c^3$ and $ca = ac^2$ 
in the list of relators is due to the fact that $\sigma_a^2 = \sigma_c^3 = ()$,
and $\sigma_c \sigma_a = \sigma_a \sigma_c^2$.}
$$
K := < a, b, c, d, e; ab = ba, de = ed, b, d, e = a, a^2, c^3, ca = ac^2 >
$$
to obtain a new group $\newgroup = \pione / \hat H$ and a projection $q : \pione \to \newgroup$,
where $\hat H$ is the normal subgroup of $\pione$ generated by the added relators.
A sequence of Tietze transformations \cite{MaKaSo:76} reduces the above presentation to
$\newgroup = <a, c; a^2, c^3, ca = a c^2>$ which is quickly seen to be isomorphic to the symmetric
group $S_3$ with representative elements
$\SSS := \{ a^\alpha c^\gamma : \alpha \in \{0,1\}, \gamma \in \{0, 1, 2\} \}
\subset \pione$.
Upon identification of the representative elements with their equivalence class, the
projection $q$ can be interpreted as a projection $q : \pione \to \SSS$.

Finally, the subgroup $H \leqslant \pione$ is defined as the set of $g \in \pione$ such that
$\gamma = 0$ if we write $q(g)$ as $q(g) = a^\alpha c^\gamma \in \SSS$.
It corresponds to all paths in $\pione$ that remain closed when lifted on $\covering$ with starting
point $\basepointone$ taken in sheet $1$.

As an example, consider the word
$w = a d^{-1} c a^{-1} c^{-1}$ (this word corresponds to looping once around the short edge
$\shortone$, see \eqref{eq:arcs}).
We can remove all occurrences of $d$ (and of $b$, but there is none anyway,
also we could substitute $a$ for $e$ if any occurrence of $e$ were present)
to obtain the
word $a c a^{-1} c^{-1}$.
Enforcing $a^2 = c^3 = 1$ (empty word) we arrive at $a c a c^2$;
using $ca = ac^2$ then produces $a^2 c^4$ that finally reduces to the
normal form $a^\alpha c^\gamma$, with $\alpha = 0$, $\gamma = 1$.
Since $\gamma \not= 0$ we conclude that $w \not \in H$.

\begin{Proposition}\label{prop:the_subgroup_has_index_3}
The subgroup $H$ has index $3$ in $\pione$ and it is not normal.
\end{Proposition}

\begin{proof}
That $H$ is a subgroup is a direct check.
Its right cosets are obtained by right multiplication by $\gamma$ and $\gamma^2$ showing
that there are exactly three cosets (they correspond to $\gamma = 0,1,2$ in $\SSS$).
It is not a normal subgroup since $a \in H$ ($a = a^1 c^0$, hence $\gamma = 0$), but $c a c^{-1} \not\in H$.
Indeed $q(c a c^{-1}) = q(a c)$ by enforcing $c a = a c^2$.
The non normality of $H$ is consistent with the non normality of the cover.
\end{proof}

The next result ensures in particular that the approach of Section \ref{sec:functional}
is independent of the choice of the cut surface. 

\begin{theorem}\label{teo:abstract}
The cover $\prjabstract : \coveringabstract \to \base$
defined by $H \leqslant \pione$ is isomorphic to the
cover $\prj : \covering \to \base$ defined with the cut and paste technique.
\end{theorem}

\begin{proof}
The proof consists in a direct check that $\pione$ defines the same action on the fiber over the base
point $\basepoint \in  \base$
\cite[p. 70]{Ha:02}.
We first need to define a bijection between the two fibers
$\prj^{-1}(\basepoint)$ and $\prjabstract^{-1}(\basepoint)$.
In view of \eqref{eq:abstractequiv} the set $\prjabstract^{-1}(\basepoint)$ consists in equivalence classes
of elements of $\pione$ with respect to the equivalence relation $g_1 \sim g_2$ if and only if
$g_1 g_2^{-1} \in H$ (we slightly abuse the notation used in \eqref{eq:abstractequiv}
for the equivalence relation and use it here on elements of $\pione$ rather than on loops based on $\basepoint$).
In other words $\prjabstract^{-1}(\basepoint)$ consists of the three right cosets of $H$ in $\pione$.
This amounts in tagging the three right cosets with the numbers $1, 2, 3$ and identifying the elements of
$\prj^{-1}(\basepoint)$ and $\prjabstract^{-1}(\basepoint)$ with the same tag.
Observe that Proposition \ref{prop:the_covering_is_not_normal} implies that there is a unique tagging that will induce
the desired isomorphism.
The three right cosets of $H$ can be described as $H$, $Hc$ and $Hc^2$ and will be tagged as $1$, $3$ and $2$ respectively.
It is now sufficient to check that the action of the generators $a, b, c, d, e$ of $\pione$ on the fibers gives the
same permutation of the tagging (recall that in the abstract construction $\pione$ acts on the fiber
$\prjabstract^{-1}(\basepoint)$ as right multiplication).
By comparing Figures \ref{fig:covering_base} and \ref{fig:covering_cut} we see that $c$ corresponds to the cyclic
permutation $(1,3,2)$ on $\prj^{-1}(\basepoint)$ and the same is true in the abstract construction in view of the
chosen tagging.
{}From the definition of $H$ we see that $Ha = He = H$ whereas $Hca = Hce = Hc^2$, corresponding to the transposition
$(2,3)$ of the tagged abstract fiber, the same as for the cut and paste construction.
Finally, the two generators $b$ and $d$ clearly act as the identify 
on both $\prj^{-1}(\basepoint)$
and $\prjabstract^{-1}(\basepoint)$.
\end{proof}

Remark \ref{rem:metric} clearly applies to this construction of the covering space, so that the isomorphism
of the covers constructed above is also an isometry.


\subsection{Structure of the branch curves.}\label{sec:completion}

The covering space $\covering$, viewed as a metric space with the metric locally induced by the
base space $\base$, can be completed into $\ccovering$ with the addition of \emph{branch curves}.
The projection $\prj$ then naturally extends to
$$
\prj : \ccovering \to \R^3
$$
which will now be a branched cover.

Of particular importance are the branch curves corresponding to Cauchy sequences that converge in
$\base$ to points belonging to the invisible loops $\loopone$ and $\looptwo$; 
their structure allows to construct functions $u : \covering \to \{0,1\}$
whose $p(J_u)$ does not wet $\iwires$, see Theorem \ref{teo:properties_of_sigma}.

As a direct consequence of the cut and paste construction of the cover, in particular of the fact
that the permutation of sheets associated to the two disks fixes sheet $1$ and swaps sheets $2$ and $3$,
we have the following

\begin{remark}
The inverse image $\prj^{-1}(\loopone)$ consists of two connected components,
$\prj^{-1}(\loopone) = \loopone^1 \cup \loopone^{23}$:
$\loopone^1$ being a curve containing no ramification points,
i.e. having a small tubular
neighborhood homeomorphic to its projection into a tubular neighborhood of $\loopone \subset \R^3$;
$\loopone^{23}$ being a ramification curve 
of index two, i.e.
having  a small tubular neighborhood that projects onto
its image as a branched cover of degree two.
Similar properties hold  for $\looptwo$.
Instead, the inverse image $\prj^{-1}(\edges)$ is connected with ramification index three.
\end{remark}

\section{The minimization problem}\label{sec:functional}
We refer to \cite{AmFuPa:00} for all details on functions of bounded variation;
we denote by $\HHH^\ell$
the $\ell$-dimensional Hausdorff measure in $\R^3$, for $\ell=1,2$.

For any specific (geometric) definition of $\base$ (and hence of $\covering$), we set
\begin{equation}\label{eq:domain0}
\domainzFb := \left\{ u \in BV(\covering; \{0,1\})
 : \sum_{y \in \prj^{-1}(x)} u(y) = 1 ~\text{for a.e. } x \in \base \right\}.
\end{equation}

Then we impose a ``Dirichlet'' boundary condition at infinity, 
and 
the domain of the
functional $\FFF$ is defined as

\begin{equation}\label{eq:domain}
\domainFb := \left\{
 u \in \domainzFb : u(y) = 1 ~\text{for a.e. } y \in \text{ sheet $1$ of } \prj^{-1}(x), |x| > C
\right\},
\end{equation}
for $C$ large enough such that the ball of radius $C$ compactly contains the solid wedge $\Wedge$.
In view of the fact that the covering is not normal (Proposition \ref{prop:the_covering_is_not_normal}),
the choice of the Dirichlet condition is now
quite important.
If there is no risk of confusion we shall often drop the dependency on $\base$ and simply write $\domainzF$ and $\domainF$ in place of
$\domainzFb$ and $\domainFb$.

Finally, the functional to be minimized is
$$
\FFF(u) =
\begin{cases}
\frac{1}{2} | Du |(\covering) \qquad & \text{ if } u \in \domainF,
\\
+\infty \qquad & \text{ otherwise. }
\end{cases}
$$
The presence of the constant $1/2$ is due to the fact that, 
if $u$ jumps at a point of a sheet, then the constraint in \eqref{eq:domain0}
forces
$u$ to jump also at the corresponding point (i.e., on the same fiber) of another sheet, 
while on the remaining sheet $u$ does not jump. $\vert Du\vert$
is the usual total variation for the scalar-valued
function $u$, and $|Du|(\covering)$ can be defined using a partition of unity
associated to a finite atlas of $\covering$ made of locally trivializing charts.

\noindent 
Given $u \in \domainzF$, we denote by 
 $J_u \subset \covering$  the jump set of $u$.

\begin{definition}
A ``film surface'' is defined as
$\prj(J_u) \subset \base$, for $u \in \domainzF$.
\end{definition}

The film surface behaves well with respect to the jump set, 
in the sense that the total variation has the following representation,
which specifies in which sense we are considering the notion of area:

\begin{theorem}
For all $u \in \domainzF$ we have
$$
\FFF(u) = \HHH^2(\prj(J_u)).
$$
\end{theorem}

\begin{proof}
It is enough to repeat the arguments of \cite[Lemma 2.12]{AmBePa:15}, 
by using local parametrizations of $Y$, and 2-rectifiability of
the jump set of a $BV$-function.
\end{proof}

\begin{definition}
For a given geometry of $\base$, 
the minimum value
of $\FFF$ depends on $\base$; we set
$$
\minFb := \inf_{u \in \domainzFb} \FFF(u)
$$
and
$$
\superficiminime(\base) := \{\prj(J_u) : u \in \domainzFb \text{ and } \FFF(u) = \minFb\} .
$$
\end{definition}

By the semicontinuity and compactness properties of $\FFF$, the infimum 
on the right hand side
is actually a minimum: this can be proven arguing as in 
\cite{AmBePa:15}.
We shall denote by $\umin = \umin(\base)$ a function such that
$\FFF(\umin) = \minFb$ and by
$\superficieminima = \prj(J_{\umin}) = \superficieminima(\base) \in \superficiminime(\base)$
the corresponding film surface ($BV$-minimizer). 
In particular the set of minimizing film surfaces $\superficiminime(\base)$ is nonempty.

\begin{theorem}[Wetting condition]\label{teo:properties_of_sigma}
Given $u \in \domainFb$, the set $\projectionofjumpofu$ satisfies the 
following properties:

\begin{itemize}
\item[\propertyone]
any closed curve that loops around a long edge $\longone$, $\longtwo$, $\longthree$ or $\longfour$ intersects
$\prj(J_u)$;
\item[\propertytwo]
any closed curve that loops around a short edge $\shortone$ or $\shorttwo$ at a distance
smaller than $\wiredistM$ intersects $\projectionofjumpofu$.
\end{itemize}
In particular,  it is possible that a closed curve around $\shortone$ or $\shorttwo$ does
not intersect $\projectionofjumpofu$.
\end{theorem}

By ``loop around an edge'' we mean that it can be continuously deformed without crossing
any edge of the tetrahedron into a ``meridian'' of the edge, a loop that orbits around that
edge alone at a small distance.

\begin{proof}
Let $x \not\in \projectionofjumpofu$ and take the precise representative of $u$ (still
denoted by $u$, it satisfies the constraint in \eqref{eq:domain0} in $\base \setminus \projectionofjumpofu$).
By construction of the covering space, a loop%
\footnote{
To get an element of $\pione$ we need to select a path from $\basepoint$ to any point of
the loop.
The element of $\pione$ consists in first moving along such path, then following the loop
and finally going backwards to $\basepoint$ along the selected path.
This does not impact the reasoning above.
}
in $\base$ based at $x$ around anyone of the long edges lifts into
a path that moves all sheets of the fiber over $x$, in particular it moves
the fiber where $u = 1$,  taking it into a fiber where $u = 0$ (condition in \eqref{eq:domain0}).
This forces $u$ to jump along the curve obtained by lifting the loop
and gives \propertyone.
Property \propertytwo\ is proved similarly by observing that a curve that loops around,
say $\shortone$ at a distance smaller than $\wiredistM$ cannot also interlace $\loopone$
and again when lifted in the covering space it moves all points of the fiber.
\end{proof}

The following definition is of central importance and highlights the essential
feature of minimizers in order to be a ``least area soap film'' for an elongated tetrahedron.

\begin{definition}[Non-wetting condition]\label{def:NW}
For a given geometry $\base$ we say that $\surfacewetting \in \superficiminime(\base)$
satisfies condition (NW) (non-wetting condition) if it does not intersect the invisible wires:
$$
\surfacewetting \cap \iwires = \varnothing .
$$
We say that the base space $\base$ satisfies condition (NW) if
$$
\surfacewetting \cap \iwires = \varnothing \qquad \forall \surfacewetting \in \superficiminime(\base) .
$$
\end{definition}

\begin{remark}\label{rem:taylor}
By compactness, if $\surfacewetting \in \superficiminime(\base)$ satisfies property (NW), there exists 
$\delta > 0$ such that Lipschitz deformations of $\surfacewetting$ in $\R^3 \setminus \edges$ which are the identity
out of a neighbourhood of $\surfacewetting$ of 
size less than
$\delta$ will not touch $\iwires$, and can be recovered as jump set of some $u \in \domainF$.
Hence $\surfacewetting$ is $({\bf M}, 0, \delta)$-minimal in the sense of F.J. Almgren \cite{Al:76}
and in particular satisfies the conditions proved by J. Taylor \cite{Ta:76} of being locally
either a minimal surface (zero mean curvature) or three minimal surfaces meeting along a curve
at $120^\circ$.
No $T$-singularity (quadruple point) can be present as a consequence of having only three sheets in the constructed
covering.
Moreover it is clear that $\surfacewetting$ is not simply connected: any closed curve that loops around the front face of
the tetrahedron along the edges, has nontrivial linking number with $\looptwo$ and therefore cannot be
shrunk to a point by deformations on $\surfacewetting$.
\end{remark}

\subsection{Estimate of \texorpdfstring{$\minFb$}{min F(M)} from below}
A crude estimate from below of $\minFb$
is a direct consequence of property \propertyone\ above, indeed
property \propertyone\ is also satisfied by the minimal
surface $\superficieskew$ that spans
the skew quadrilateral defined by the long edges $\longs_i$, $i=1,2,3,4$
(Figure \ref{fig:surfacefilmskew}). Hence
\begin{equation}\label{eq:first_crude_estimate_from_below}
\minFb \geq 
\HHH^2(\superficieskew).
\end{equation}

\begin{figure}
\includegraphics[width=0.9\textwidth]{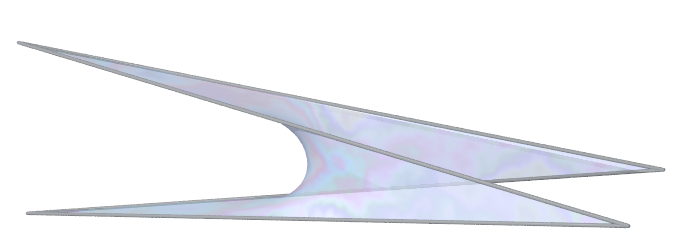}
\caption{\small{The minimal film spanning the long edges of $\tetrahedron$
($h = 3.5$).}}
\label{fig:surfacefilmskew}
\end{figure}

\begin{remark}
We have 
\begin{equation}\label{eq:bound_of_Sigma_skew}
2h
\leq
\HHH^2(\superficieskew)
\leq
\sqrt{4 h^2 + 1} + 1 .
\end{equation}
The lower bound can be obtained by reasoning as in Theorem \ref{teo:lowerbound} below, whereas
the upper bound is the area of the surface obtained by taking the upper and lower
faces of $\tetrahedron$ for $x < 0$,
the central (unit) square $\tetrahedron \cap \{x = 0\}$ and
the front and back faces for $x > 0$.
\end{remark}
Set $\base = \basehs$ for a given choice of $h$ and $s$.
For any $u \in \domainF$, 
the projection $\projectionofjumpofu$ of the jump set of $u$
satisfies properties \propertyone\ and \propertytwo\ of Theorem \ref{teo:properties_of_sigma}.
This allows us to obtain an estimate from below of its $\HHH^2$ measure,
which  refines estimates \eqref{eq:first_crude_estimate_from_below}-\eqref{eq:bound_of_Sigma_skew}, 
see formula \eqref{eq:better_estimate_from_below}.

For $t \in [-1,1]$ take the plane $\pi_t = \{x = ht\}$.
Its intersection $\Rectangle_t$ with the wedge $\Wedge$ is a rectangle of sides $1+t$ and $1-t$.
We shall derive an estimate from below of $\HHH^1(\projectionofjumpofu \cap \pi_t)$.

\subsubsection{Case $|t| > s$}
Since $\wiredistM = h(1-s)$, we have that the $L^\infty$-distance of $\pi_t$ from $\shorttwo$ if $t > 0$
(resp. $\shortone$ if $t < 0$) is
less than $\wiredistM$.
As a consequence, any curve in the rectangle $\Rectangle_t$ that connects its two long sides can be closed as a loop
around $\shorttwo$ (resp. $\shortone$) at an $L^\infty$-distance smaller than $\wiredistM$ and, in view of \propertytwo\
of Theorem \ref{teo:properties_of_sigma}, is forced to intersect
$\projectionofjumpofu\cap \pi_t$.
This, together with the first property, is enough to conclude\footnote{
We have a ``minimal partition'' problem for the rectangle $R_t$ into $3$ sets, say
$A$, $B$, $C$ with the requirement that one long edge is contained in $\overline A$,
the opposite long edge is contained in $\overline B$ and both short edges are contained
in $\overline C$.
Since the local structure of a minimizer (boundary of an optimal partition) must satisfy the properties of
a Steiner network, we only have a finite (and very small) set of possible configurations to consider.}
%
that the length
of $\projectionofjumpofu \cap \pi_t$ cannot be less than both the length
of the Steiner tree joining the four vertices of $\Rectangle_t$ and twice the
length of the long sides of $\Rectangle_t$.
Hence
\begin{equation}
\label{eq:bound0_using_Steiner}
\begin{aligned}
\HHH^1(\projectionofjumpofu \cap \pi_t) \geq &
\min \{
  1 + \sqrt{3} - (\sqrt{3} - 1) |t|,
  2 + 2|t|
  \} 
\\
= &
 \begin{cases}
  2 + 2|t| & \text{if } |t| < 2 - \sqrt{3} ,
 \\
  1 + \sqrt{3} - (\sqrt{3} - 1) |t| & \text{if } |t| \geq 2 - \sqrt{3}.
 \end{cases}
\end{aligned}
\end{equation}

\subsubsection{Case $|t| \leq s$}
We can still enforce \propertyone\ of Theorem \ref{teo:properties_of_sigma}:
any curve in $\Rectangle_t$ connecting two adjacent sides can be completed into a loop around
one of the long edges $\longs_i$, $i \in \{1,2,3,4\}$ and hence it must intersect $\projectionofjumpofu \cap \pi_t$.
It follows that the size of $\projectionofjumpofu \cap \pi_t$ cannot be less than twice the lenght of the
short sides of $\Rectangle_t$:
\begin{equation}\label{eq:bound1_using_sides}
\HHH^1(\projectionofjumpofu \cap \pi_t) \geq 2 - 2t .
\end{equation}

\begin{theorem}\label{teo:lowerbound}
For a given choice of $h$ and $s$, we have:
\begin{equation}\label{eq:better_estimate_from_below}
\minFbhs \geq
\begin{cases}
2h( 4 - \sqrt{3} - 2s^2) & \text{if } s < 2 - \sqrt{3} ,
\\
h[3 + \sqrt{3} - 2(\sqrt{3}-1)s - (3 - \sqrt{3})s^2] & \text{if } s \geq 2 - \sqrt{3} .
\end{cases}
\end{equation}
\end{theorem}

\begin{proof}
Let $u \in \domainF$.

\textit{Case $s < 2 - \sqrt{3}$.}
Using the tangential coarea formula
\cite[Theorem 3, pag. 103]{GiMoSo:98} and the sectional estimates \eqref{eq:bound0_using_Steiner}
and \eqref{eq:bound1_using_sides}, we have
$$
\begin{aligned}
\HHH^2(\projectionofjumpofu) \geq &
 \int_{-h}^h \HHH^1(\projectionofjumpofu \cap \pi_t) ~dt
\\
 \geq &
  2h \int_0^s ( 2 - 2t ) ~dt
\\
 \phantom{ppp} & + 2h \int_s^{2-\sqrt{3}} ( 2 + 2t ) ~dt
\\
 \phantom{ppp} & + 2h \int_{2-\sqrt{3}}^1 \left[ 1 + \sqrt{3} - (\sqrt{3} - 1) t \right] ~dt
\\
 = &
  2h[(2s - s^2) + (11 - 6\sqrt{3} - s^2 - 2s) + (5\sqrt{3} - 7)]
\\
 = &
  2h( 4 - \sqrt{3} - 2s^2).
\end{aligned}
$$
\textit{Case $s \geq 2 - \sqrt{3}$.}
The intermediate integral now disappears.  We get
$$
\begin{aligned}
\HHH^2(\projectionofjumpofu) \geq &
 \int_{-h}^h \HHH^1(\projectionofjumpofu \cap \pi_t) ~dt
\\
 \geq &
  2h \int_0^s ( 2 - 2t ) ~dt
\\
 \phantom{ppp} & + 2h \int_s^1 \left[ 1 + \sqrt{3} - (\sqrt{3} - 1) t \right] ~dt
\\
 = &
  h[3 + \sqrt{3} - 2(\sqrt{3}-1)s - (3 - \sqrt{3})s^2].
\end{aligned}
$$
\end{proof}

Note that for $s \to 0^+$ we obtain $\HHH^2(\projectionofjumpofu) \geq 2h(4 - \sqrt{3})$
and for $s \to 1^-$ we obtain $\HHH^2(\projectionofjumpofu) \geq 2h$ (compare
with \eqref{eq:bound_of_Sigma_skew}).

\section{A positive genus surface beating the conelike configuration}\label{sec:comparison}
For a given $h > 0$ and $0 < s < 2-\sqrt{3}$ we stick here with the choice of
$\Wedge$ (a solid elongated tetrahedron) given by Definition \ref{def:geometry},
i.e. with vertices having coordinates as in \eqref{eq:coordinates_vertices}, 
see Figure \ref{fig:surface2}, the regular tetrahedron corresponding to the choice $h = \frac{\sqrt{2}}{2}$.

Let us
denote by $\superficieconica$ a
``conelike'' film surface spanning the one-skeleton 
of $\Wedge$.

\begin{definition}[Cone-like surface]\label{def:conelike}
By ``conelike set'' (or ``conelike film surface'' if it has the appearance of a soap film)
spanning the edges $\edges$ of $\Wedge$ we mean a set
$\superficieconica$ that in $\Wedge$ separates the
four faces, i.e. such that it must intersect any path starting on one face, travelling in the interior of
$\Wedge$ and terminating on another face.
\end{definition}

The name ``cone-like'' is justified by the fact that we espect a minimal film with such
separation property to be a deformed version of the minimizing cone of Figure \ref{fig:morgan} (left).
In particular, a simply connected set in $\Wedge$ containing $\edges$ must separate the faces.
Indeed, by contradiction if it does not separate the faces we can construct a closed path disjoint
from the set that interlaces the path along
the edges of a face.
The set would then be non-simply connected (in particular non-contractible).

In Figures \ref{fig:morgan} left and Figure \ref{fig:surfacefilmc} we find two examples for a
regular tetrahedron and an elongated tetrahedron.\footnote{Note that for an elongated tetrahedron,
an area-minimizing $\superficieconica$ does {\it not} 
satisfy the usual property of cones, of being invariant under 
multiplication $x \to r x$ for $r>0$, see the caption of 
Figure \ref{fig:surfacefilmc}. This is the reason for calling 
$\superficieconica$ a conelike configuration, and not simply a cone.} 
We shall compare $\superficieconica$
with a
particular competitor $\particularsurface$ corresponding to (i.e., being the projection of) 
the jump set of a $BV$ function $u$ in the domain 
of the functional $\mathcal F$ (Theorem \ref{teo:there_exists_u_that_projects});
the competitor $\particularsurface$ will be non-simply connected.

We shall show that there exists
$h > 1$ sufficiently large such that the area of $\particularsurface$ is less than 
the area of $\superficieconica$ (Theorem \ref{teo:comparison}), giving quite strong
evidence that $\superficieconica$ is not
area-minimizing among minimal films if we allow for a more complex topology.

\subsection{Constructing the surface \texorpdfstring{$\particularsurface$}{Sigma2} using the triple cover}\label{sec:sigma2}

Let $\tau \in (s,1)$ be a parameter to be chosen later, see \eqref{eq:tau_small_enough}.
The competitor  $\particularsurface$ is constructed by joining five pieces,
\begin{equation}\label{eq:particularsurface}
\particularsurface = \Sigma_{1} \cup \Sigma_{2} \cup \Sigma_{3} \cup \Sigma_{4}
\cup \Sigma_{v},
\end{equation}
the first four
 obtained by sectioning the wedge with the three planes
$\{x = 0\}$, $\{x = \pm h\tau\}$,
see Figure \ref{fig:surface2}, and the last one being ``vertical'',
as follows:

{\it Case $x \in (-h, -h\tau)$ and $x \in  (h\tau, h)$}: the surface $\Sigma_{i}$, $i \in \{1,4\}$
is chosen coincident with $\superficieconica$,
more precisely
$\Sigma_{1} := \superficieconica \cap \{ x < -h\tau \}$ and 
$\Sigma_{4} := \superficieconica \cap \{ x > h\tau \}$.

{\it Case $x \in  (0,h\tau)$}: the surface 
$\Sigma_{3}$ coincides with the top and bottom faces of $\Wedge$.

{\it Case $x \in (-h\tau, 0)$}: the surface $\Sigma_{2}$ 
coincides with the front and back faces of $\Wedge$.

In order to close the surface we need to add three ``vertical'' pieces,
cumulatively denoted by $\Sigma_{v}$ (see Figure \ref{fig:sections}), union of 
the square
obtained by intersecting $\Wedge$ with the vertical plane $\{x=0\}$ and 
of the parts of the
two rectangles resulting as the intersection of $\Wedge$ with
the two planes $\{x = \pm h\tau\}$.
\begin{figure}
\includegraphics[width=0.9\textwidth]{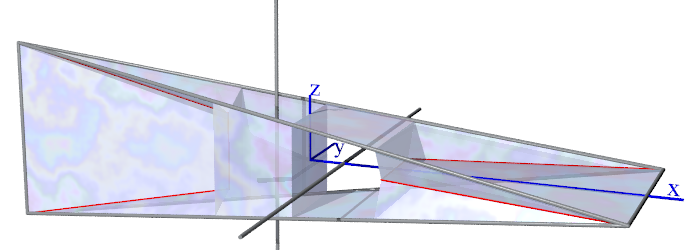}
\caption{\small{Sketch of $\particularsurface$ with $\tau = 2 - \sqrt{3}$ and $h = 3.5$.}}
\label{fig:surface2}
\end{figure}

\begin{figure}
\includegraphics[width=0.6\textwidth]{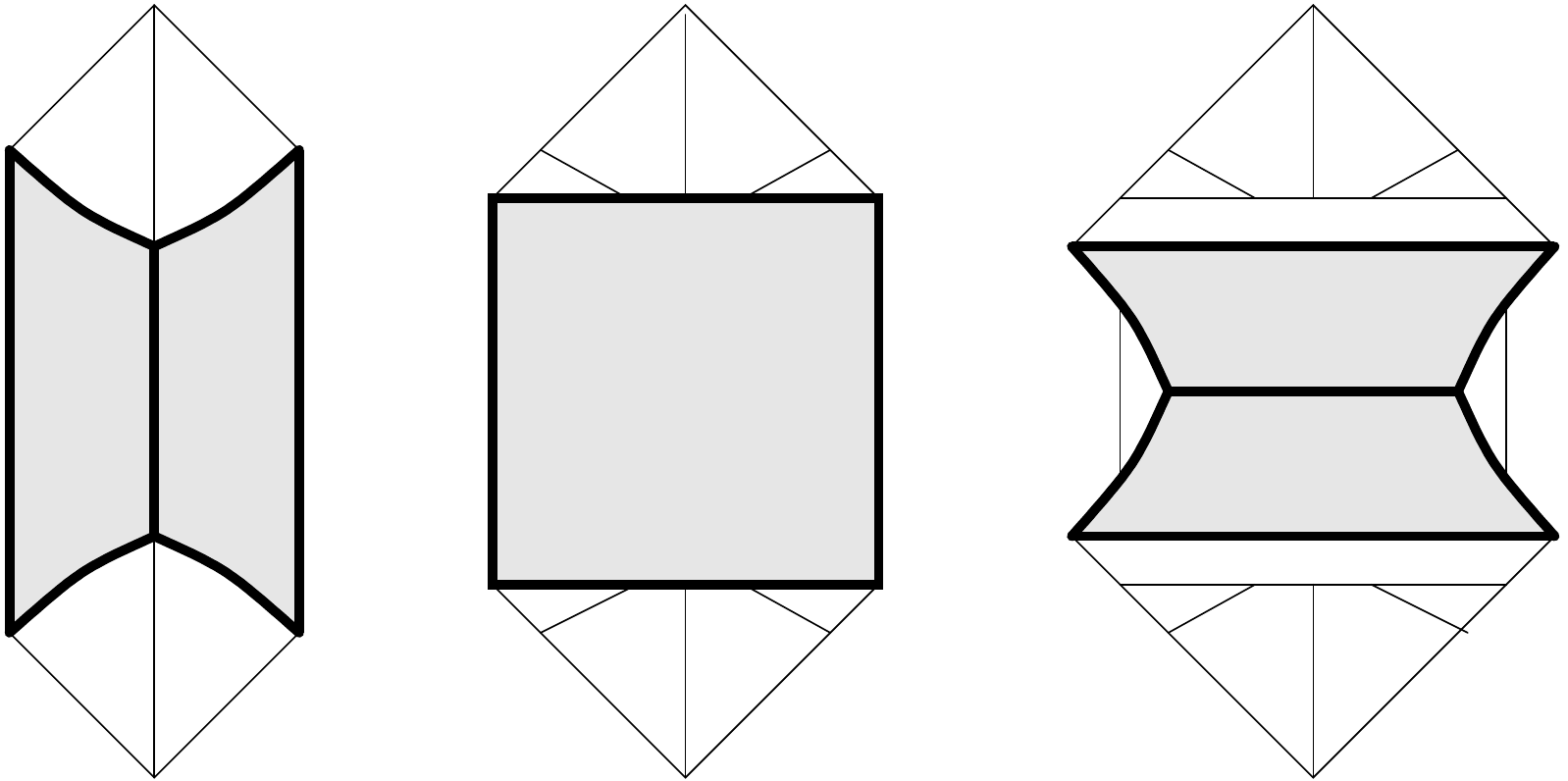}
\caption{\small{In grey the three components of $\Sigma_{v}$, vertical sections of $\particularsurface$ at
$x = -h\tau$ (left), $x = 0$ (center), $x = h\tau$ (right). The area of
these grey regions is estimated from above by their convex envelopes (see
\eqref{eq:stima_area_parte_verticale}).
}}
\label{fig:sections}
\end{figure}

\begin{theorem}\label{teo:comparison}
Let $s \in (0,2-\sqrt{3})$. If  $\tau \in (s,1)$ is small enough
 and $h \in (1,+\infty)$ is large enough depending on $\tau$,
we have
\begin{equation}\label{eq:comparison}
\HHH^2(\particularsurface) < \HHH^2(\superficieconica).
\end{equation}
\end{theorem}

\begin{proof}
We have to show that 
\begin{equation}\label{eq:what_we_have_to_show}
\HHH^2(\Sigma_1) +  
\HHH^2(\Sigma_2) + \HHH^2(\Sigma_3) + \HHH^2(\Sigma_4) + \HHH^2(\Sigma_v) < \HHH^2(\superficieconica),
\end{equation}
provided $\tau \in (s,1)$ is small enough and $h = h(\tau)\in (1,+\infty)$ is large enough.

In view of the definition of $\particularsurface$ we have
$$
\HHH^2(\superficieconica \cap \{ |x| > h\tau \}) = \HHH^2(\particularsurface \cap \{ |x| > h\tau \}),
$$
and therefore inequality \eqref{eq:what_we_have_to_show} is equivalent to 
\begin{equation}\label{eq:what_equivalently_we_have_to_show}
\HHH^2(\Sigma_2) + \HHH^2(\Sigma_3) + \HHH^2(\Sigma_v) < 
\HHH^2(\superficieconica\cap \{-h\tau < x < h\tau\}).
\end{equation}
As in Section \ref{sec:functional}, 
for $t \in [0,1]$,
the intersection 
of the plane 
 $\pi_t = \{x = ht\}$ 
 with the wedge $\Wedge$ is a rectangle of sides $1+t$ and $1-t$.
Since $\superficieconica$ divides $\Wedge$ into four disjoint solid regions,
 one
per face, it follows that $\superficieconica \cap \pi_t$ divides the rectangle
into four disjoint regions.
Hence
\begin{equation}\label{eq:bound_using_Steiner}
\HHH^1(\superficieconica \cap \pi_t) \geq 1 + \sqrt{3} - (\sqrt{3} - 1) t,
\end{equation}
the right hand side 
being the length of the Steiner tree joining the four
vertices of the rectangle.

For a given $\tau \in (s,1)$ we shall need a 
bound from below of the section
$\superficieconica \cap \{ -h\tau < x < h\tau \}$:
using the coarea formula 
and \eqref{eq:bound_using_Steiner}, we have 
$$
\begin{aligned}
\HHH^2(\superficieconica\cap \{-h\tau 
< x < h\tau\}) 
) \geq & \int_{-h}^h \HHH^1(\superficieconica\cap \{-h\tau 
< x < h\tau\}
\cap \pi_t) ~dt
\\
 \geq & 
2 h \int_0^{\tau} \left(1 + \sqrt{3} - (\sqrt{3} - 1) t\right) ~dt
\\
= &
2(1 + \sqrt{3}) h \tau - (\sqrt{3} - 1) h \tau^2.
\end{aligned}
$$
Therefore, in order to show \eqref{eq:what_equivalently_we_have_to_show} it is sufficient to prove
\begin{equation*}
\HHH^2(\Sigma_2) + \HHH^2(\Sigma_3) + \HHH^2(\Sigma_v) < 
2(1 + \sqrt{3}) h \tau - (\sqrt{3} - 1) h \tau^2.
\end{equation*}

Since all intersection rectangles have the same perimeter and the central square has area equal to one, we
have 
\begin{equation}\label{eq:stima_area_parte_verticale}
\HHH^2(\Sigma_{v}) \leq 3,
\end{equation}
and so it will be sufficient to prove
\begin{equation}\label{eq:what_in_turn_will_be_sufficient}
\HHH^2(\Sigma_2) + \HHH^2(\Sigma_3) + 3 < 
2(1 + \sqrt{3}) h \tau - (\sqrt{3} - 1) h \tau^2.
\end{equation}

Now, the area of the top (or bottom)  facet $F$ of $\Wedge$ (the one
having the vertex on the left and the basis on the right) equals
 $\sqrt{4 h^2 + 1}$, therefore
$$
\begin{aligned}
\HHH^2(F \cap \{x < h\tau\})= &
\frac{(1 + \tau)^2}{4} \sqrt{4 h^2 + 1},
\\
\HHH^2(F \cap \{0 < x < h\tau\})
= &
\frac{(1 + \tau)^2}{4} \sqrt{4 h^2 + 1} - \frac{1}{4} \sqrt{4 h^2 + 1}.
\end{aligned}
$$
It follows
$$
\HHH^2(\Sigma_{2}) + \HHH^2(\Sigma_{3}) 
= 
4 \HHH^2(F \cap \{0 < x < h\tau\}) =(2 + \tau) \tau \sqrt{4 h^2 + 1},
$$
so that \eqref{eq:what_in_turn_will_be_sufficient} will be proved if we show 
$$
{\rm L}:= \frac{1}{h \tau}\left[(2 + \tau) \tau \sqrt{4 h^2 + 1} + 3\right]
<
 2(1 + \sqrt{3})  - (\sqrt{3} - 1) \tau =: {\rm R}.
$$
Let us select $\tau \in (s,1)$ sufficiently small so that  
\begin{equation}\label{eq:tau_small_enough}
4 + 2 \tau < 2(1+\sqrt{3}) - (\sqrt{3}-1) \tau ,
\end{equation}
one possibility is e.g. to choose $\tau = 2 - \sqrt{3}$, consistent with $\tau \in (s,1)$ in
view of the constraint imposed  on $s$.
Then we have
$$
\lim_{h \to +\infty} \textrm{L} = 4 + 2\tau < \textrm{R},
$$
and the result follows.
\end{proof}

Inequality \eqref{eq:tau_small_enough} is solved for
$0 < \tau < 2(2-\sqrt{3})$.
Values leading to inequality \eqref{eq:comparison} are e.g.
$$
\tau = 2-\sqrt{3}, \qquad h = 16 ,
$$
they lead to the values 
$$
\HHH^2(\particularsurface) \approx 22.456 + c, \qquad 
\HHH^2(\superficieconica) \approx 22.585 + c
$$
where $c$ is the common value
$$
c := \HHH^2(\superficieconica \cap \{ |x| > h\tau \}) = \HHH^2(\particularsurface \cap \{ |x| > h\tau \}),
$$

\begin{figure}
\includegraphics[width=0.9\textwidth]{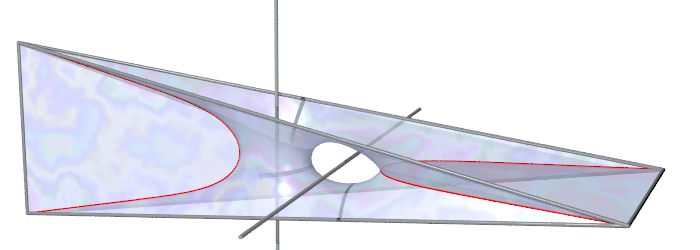}
\caption{\small{Minimal film obtained numerically starting a gradient flow from the $\particularsurface$ in \eqref{eq:particularsurface}, with
$\tau = 2 - \sqrt{3}$ and $h = 3.5$.
It satisfies property (NW) of Definition \ref{def:NW}.
Computation performed with the \texttt{surf} software code by Emanuele Paolini.}}
\label{fig:surfacefilm}
\end{figure}

\begin{figure}
\includegraphics[width=0.9\textwidth]{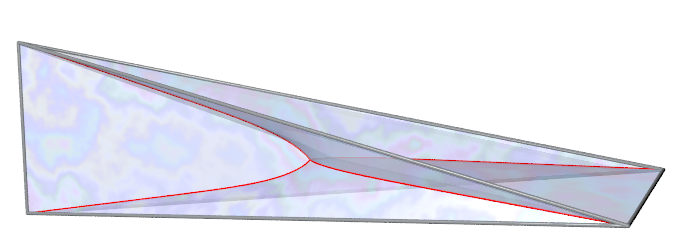}
\caption{\small{The conelike configuration $\superficieconica$ with $h=3.5$.
Note carefully that only the two crescent-like surfaces lying in $\{y = 0\}$ on the
left and in $\{z = 0\}$ on the right
are flat. 
Computation performed with the \texttt{surf} software code by Emanuele Paolini.}}
\label{fig:surfacefilmc}
\end{figure}

\begin{figure}
\begin{overpic}[width=0.75\textwidth]{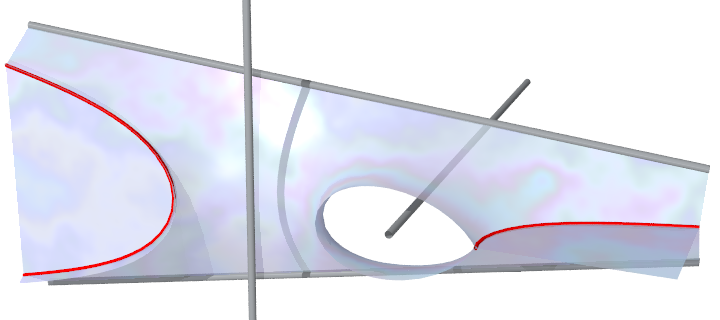}
\put (80,35) {\large [clipped]}
\end{overpic}
\caption{\small{A zoom of the computed minimal surface of Figure \ref{fig:surfacefilm} clipped at $y=-0.05$;
the $120^\circ$ condition can be clearly seen where the horizontal tunnel meets the triple line.}}
\label{fig:surfaceclipped}
\end{figure}

\begin{theorem}\label{teo:there_exists_u_that_projects}
For any choice of $s \in (0, 2-\sqrt{3})$
select the base space $\base = \basehs$ in Definition \ref{def:geometry}
with $h$ large enough (e.g. $h > 16$).
Let 
$\particularsurface$ be the competitor defined in \eqref{eq:particularsurface}.
Then there exists $u \in \domainF$ such that 
 $\projectionofjumpofu = \particularsurface$. In particular, 
$\FFF(u) = \HHH^2(\particularsurface)$, and
$$
\minFb < \HHH^2(\superficieconica).
$$
\end{theorem}

\begin{proof}
Fix $\tau = 2 - \sqrt{3}$ and let $\particularsurface$ be the corresponding competitor 
defined in \eqref{eq:particularsurface}.
Since $s < 2 - \sqrt{3} = \tau$, it follows that the invisible wires do not
intersect $\particularsurface$ and
in view of Remark \ref{rem:cut_deformation} we can assume that the cutting set in the cover construction
is defined by $\particularsurface$ and then simply define $u : \covering \to \{0,1\}$
as $1$ on the first sheet and zero otherwise.
Clearly $u$ satisfies the required constraints to ensure $u \in \domainF$ and
in view of the gluing permutations (in particular $u$ is locally constant, $u = 1$ on sheet $1$ and $u = 0$ on sheets $2$ and $3$,
on a neighborhood of $\prj^{-1}(\iwires)$, compare Figure \ref{fig:covering_cut})
it is easy to show that $\projectionofjumpofu = \particularsurface$,
\end{proof}

Function $u$ can also be constructed directly using the abstract definition of the covering
(Section \ref{sec:abstract}) as follows.
First we need to fix an orientation and decide a permutation $\sigma \in S_3$ for each smooth portion of
$\particularsurface$.
This is done consistently to the permutations of Figure \ref{fig:covering_cut}, for example
we associate the permutation $(1,3,2)$ to the left flat ``lunette'' when traversing it from front to back.
We also need two ``phantom'' disks mimicking the cutting disks of figure \ref{fig:covering_cut} (with associated
permutation $(2,3)$) that cut e.g. the frontal trapezium with a vertical line
in a right part with permutation $(1,3,2)$ and a left part with permutation $(1,2,3)$ (when traversing it
from front to back).
The central vertical square would have the permutation $(1,3,2)$ associated to it when traversing from right to
left.
In a similar fashion we attach a suitable permutation to all the remaining (oriented) portions of the surface $\particularsurface$, taking
into account that the top trapezium is also divided in two parts by the phantom disk.

Now we first define a function $\hat u$ on the set $\ucovering$
of paths in $\base$ starting at the base point $\basepoint$.
If $\gamma$ is such a path with $x := \gamma(1) \not\in \particularsurface$, we can suppose, up to a small deformation
in the same homotopy class,
that it has only trasversal intersections with $\particularsurface$ and no intersections with the
triple curves nor with the intersection of the phantom disks with $\particularsurface$.
Then we can enumerate the permutations associated to the intersections of $\gamma([0,1])$ 
with $\particularsurface$
and the phantom disks or their inverse (based on whether $\gamma$ traverses the surface in a positive or negative
direction with respect to its selected orientation) and multiply all these permutations to obtain $\sigma_\gamma \in S_3$.
If the final permutation fixes $1$, i.e. $\sigma_\gamma(1) = 1$, then we define $\hat u(\gamma) = 1$, otherwise $\hat u(\gamma) = 0$.

The desired function $u : \covering \to \{0,1\}$ is now defined as $u([\gamma]) = \hat u(\gamma)$ where $[\gamma]$ is
the equivalence class of  $\gamma$ in \eqref{eq:abstractequiv}.
It is necessary to show that this is a good definition, in other words, that
$\hat u (\gamma_1) = \hat u(\gamma_2)$ whenever $\gamma_1 \sim \gamma_2$,
i.e. whenever $[\gamma_1 \gamma_2^{-1}] \in H$.
It is readily seen that this is a consequence of the stronger requirement that $\hat u(\gamma) = 1$ for all
$\gamma$ closed curve 
with $[\gamma] \in H$, which we now prove.

The choice of permutations on the pieces of surface is chosen such that the final permutation computed on a closed
$\gamma$ is insensitive to homotopic deformations of $\gamma$, so that we only need to show that $\sigma_\gamma$
fixes $1$ whenever $[\gamma] \in H$.
This is true precisely because the choice of the permutations mimics the permutations used to define the covering
by cut and paste displayed in Figure \ref{fig:covering_cut}.

\begin{remark}
Inequality \eqref{eq:comparison} is crucial in trying to actually prove the existence
of a non-simply connected minimal film spanning an elongated tetrahedral frame,
since it shows the existence of a surface
with the desired topology having area strictly less than the minimal area achievable with
conelike configurations.
The candidate would be a minimizer of $\FFF$, since Theorem \ref{teo:there_exists_u_that_projects} implies that
$\minFb < \HHH^2(\superficieconica)$.
However we still are unable to conclude, 
 because we cannot exclude that the minimizing surface interferes
with the invisible wires, i.e. it does not satisfy property (NW) of Definition \ref{def:NW}
(see Section \ref{sec:whatcangowrong}).
Numerical simulations however strongly suggest that with appropriate choice of $h$ and $s$ in $\basehs$
this is not the case (Figure \ref{fig:surfacefilm}).
\end{remark}

\subsection{Comparison with the Reifenberg approach}\label{sec:reifenberg}
The approach of E. R. Reifenberg \cite{Re:60} to the Plateau problem is based on \v Cech
homology. We want to show here that, presumably, the Reifenberg approach
(in three space dimensions 
and in codimension one) cannot
reproduce a surface with the topology as the one depicted in Figure \ref{fig:morgan}, right.

One first fixes a compact\footnote{See \cite{Fa:16} for an extension of the theory for a noncompact
$G$.}
 abelian group $G$ (for our purposes it is convenient to think of $G$
as if $G = \Z$ even if this is not compact; in what follows the choice $G = \Z_m$ with
various values for $m$ leads to the same considerations).
In the sequel all the homology groups are isomorphic to the direct sum $\bigoplus_{i=1}^r G$
of $r$ copies of $G$, we shall refer to $r$ as the \emph{rank} of the homology group.

Next,
given a compact subset
$\boundaryReifenberg$ of $\R^3$, one has to minimize the Hausdorff
measure $\HHH^2(\surfaceReifenberg)$ of $\surfaceReifenberg$,
 among all compact sets $\surfaceReifenberg \supseteq \boundaryReifenberg$ in
$\R^3$ satisfying a suitable condition, that we will specify.
Here we fix $\boundaryReifenberg$ to be the union of the six edges of a tetrahedron.

The homology group $H_1(\boundaryReifenberg;G)$ is seen to have rank $3$, by
observing that $\boundaryReifenberg$ is homotopic to a bouquet of three loops, and a convenient choice
of the generators is:

\begin{itemize}
\item[$\alpha$:] (counterclockwise) loop around the front face, described as $\shortone^{-1} \longfour^{-1} \longtwo$
with reference to Figure \ref{fig:wedge};
\item[$\beta$:] loop around the top face, $\longtwo^{-1} \shorttwo \longone$;
\item[$\ell:$] loop along the long edges, $\longfour^{-1} \longtwo \longone^{-1} \longthree$.
\end{itemize}

For some $\surfaceReifenberg \supseteq \boundaryReifenberg$
the inclusion $i : \boundaryReifenberg \to \surfaceReifenberg$
induces a homomorphism $i_* : H_1(\boundaryReifenberg;G) \to H_1(\surfaceReifenberg;G)$ between the first homology groups of $\boundaryReifenberg$ and $\surfaceReifenberg$ respectively,
whose kernel is called algebraic boundary of $\surfaceReifenberg$.

At this point for a given subgroup $L < H_1(\boundaryReifenberg;G)$ we search for a minimizer of
$\HHH^2(\surfaceReifenberg)$ among all $\surfaceReifenberg \in \SSS(L)$ where
$\SSS(L)$ is the family of compact sets whose algebraic boundary contains $L$.

If $\psurfaceReifenberg$ is a surface with the required topology (e.g. the one of Figure \ref{fig:surfacefilm}
or the one of Figure \ref{fig:surface2}, or even that of Figure \ref{fig:morgan} right),
we want on one hand $\psurfaceReifenberg \in \SSS(L)$ and on the other hand
$\SSS(L)$ to be as small as possible,
which leads to the choice $L = \text{ker}(i_*)$.

The set $\psurfaceReifenberg$ has first homology group
$H_1(\psurfaceReifenberg,G)$ of rank two, generated by $i_*(\alpha)$ and $i_*(\beta)$, whereas
$\ell$ is a generator of the kernel of $i_*$, leading to the choice of $L$ as the subgroup of
$H_1(\boundaryReifenberg,G)$ generated by $\ell$.

The family $\mathcal{S}(L)$ then contains subsets
of $\R^3$ with first homology group of rank $2$, containing $\boundaryReifenberg$ and with algebraic boundary $L$.

Unfortunately the imposed condition on the algebraic boundary does not impose wetting of the two short
edges $\shortone$ and $\shorttwo$, and indeed the surface of Figure \ref{fig:surfacefilmskew} also
is in $\mathcal{S}(L)$ and we presume it to be the Reifenberg minimizer.


\section{Positioning the invisible wires}\label{sec:whatcangowrong}
%
%
For a fixed (sufficiently large) choice  of $h$ the minimum value $\minFb=\FFF(\umin) $
will depend on the relative position of the invisible wires $\iwires$ with respect to the
tetrahedral frame.
Our first guess would be that for a wide range of positions (those for which $\iwires$ does not touch $\superficieminima=\prj(J_{\umin}) 
$) such value is constant, and so is a minimizer of the functional.

When $\iwires$ leaves such a set of positions we would expect the minimum value to increase a bit, since
in that case the invisible wires impose a further constraint on $\superficieminima$.
Indeed the wires would ``push'' on the film surface and act as an obstacle for as long as the deformed
surface bends at the wire with an angle larger than $120$ degrees.
This behaviour minics the situation of a Steiner tree for three points vertices of an obtuse triangle with
an angle larger than $120^\circ$.

Beyond the $120^\circ$ threshold we expect one of the local ``J. Taylor'' rules \cite{Ta:76}  for a minimizing film to
take effect and observe the formation of a new (fin-like) portion of the surface connecting a portion
of $\iwires$, let us call it the ``wetted portion'', to a triple curve on the deformed surface meeting
at angles of $120^\circ$.

\begin{figure}
\includegraphics[height=0.5\textwidth]{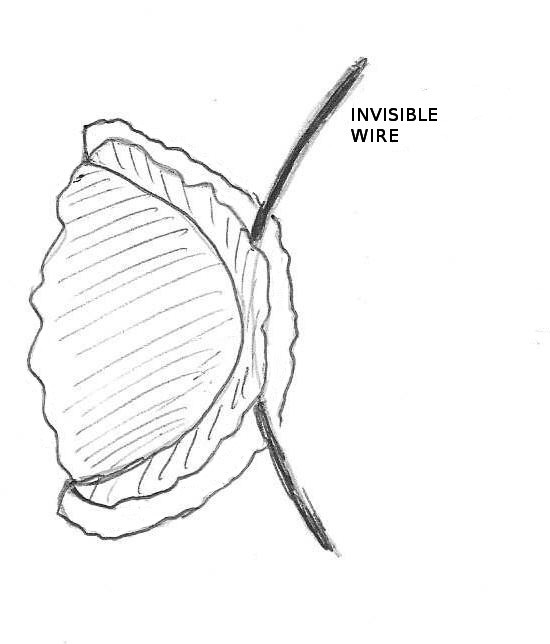}
\includegraphics[height=0.5\textwidth]{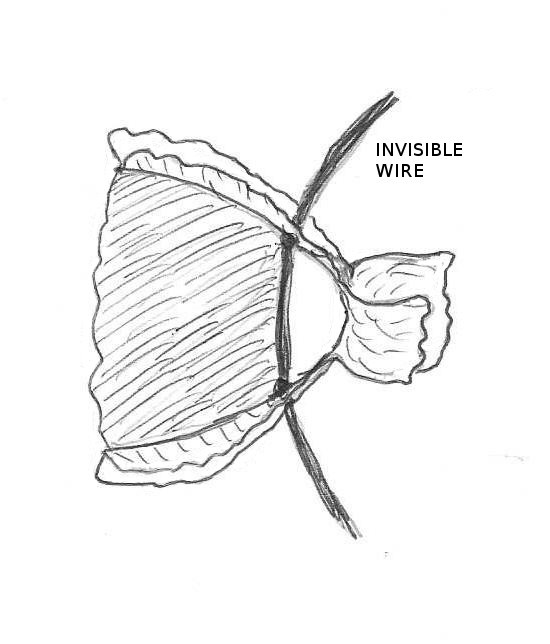}
\caption{\small{When the invisible wire approaches the triple curve (left) it becomes energetically
convenient for the surface to jump into a configuration (right) where the invisible wire gets partially
wetted (thus no longer invisible!) and a new hole is created right of the wire.}}
\label{fig:wetting}
\end{figure}

\begin{figure}
\includegraphics[width=0.9\textwidth]{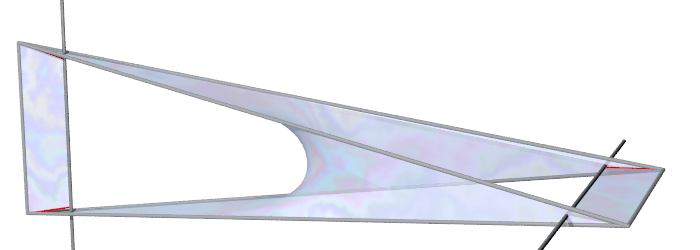}
\caption{\small{If the invisible wires approach the two short edges, the  minimizer takes the
structure shown in the picture ($h = 3.5$).  The invisible wires are partially wetted by the film.
Four short triple curves join the vertices of the frame with the boundaries of the wetted portion of
the invisible wires.
The structure of the film surface is different from that of Figure \ref{fig:surfacefilm}, in
particular it does not satisfy property (NW) of Definition \ref{def:NW}.
}}
\label{fig:surfacefilmwrong}
\end{figure}

The story is however completely different if $\iwires$ is moved to meet one or both of the two
triple curves of the minimizing surface (red curves in Figure \ref{fig:surfacefilm}). 

In this situation it is energetically favorable for the surface to suddenly jump into a configuration
where a large portion of $\iwires$ is wetted by the flat part of the minimizer (Figure \ref{fig:wetting}).
Two new holes 
in the surface would then be created.

Actually this would be even more dramatic, since the formation of two smooth catenoid tunnels would
come out in a situation where the tunnels are too long to be stable, so we also expect the tunnels
to disappear completely with a final configuration resembling the one that would be obtained by a
film that does not wet $\shortone$ and $\shorttwo$ with the addition of two flat trapezoid portions connecting e.g.
$\shortone$ with part of $\loopone$ and with the rest of the surface (similarly on the right),
see Figure \ref{fig:surfacefilmwrong}.

In order to rule out this possible minimizer we can derive a lower bound for the surface area in such a
configuration.

\begin{definition}\label{def:Steiner_like_surface}
For a given choice of $h>0$ and $0 < s < 1$ and selecting $\base = \basehs$ we say that
a film surface 
$\genericsurfaceSteiner$ is ``Steiner-like'' if it satisfies properties
\propertyone, \propertytwo\ of Theorem \ref{teo:properties_of_sigma}\footnote{When $\projectionofjumpofu$ is replaced by $\genericsurfaceSteiner$.} and moreover the intersection
$\genericsurfaceSteiner \cap \pi_t$ with $\pi_t = \{x = ht\}$ and $|t| > s$ separates the four sides of the
rectangle $\tetrahedron \cap \pi_t$.
\end{definition}

\begin{theorem}\label{teo:modification_of_a_Steiner_like_surface}
Let $s \in (0, 2 - \sqrt{3})$. 
A Steiner-like surface $\superficiesteiner$ can be modified into a surface $\particularsurfaceSteiner$ with
the topology of
the one constructed in Section \ref{sec:sigma2} that does not wet the invisible wires and has lower
area provided we choose $s \leq s_0$, $s_0$ small enough and then $h$ large enough.
Consequently $\superficiesteiner$ cannot be a minimizer of functional $\FFF$.
\end{theorem}

\begin{proof}
The proof mimics that of Theorem \ref{teo:comparison}. We modify $\superficiesteiner$ in the region
$-h\tau < x < h\tau$ with the choice $\tau = 2 - \sqrt{3}$ exactly as we did for Theorem \ref{teo:comparison}
obtaining a surface $\particularsurfaceSteiner$.
Using again the coarea formula,
the sectional estimate \eqref{eq:bound0_using_Steiner}
with the second case also if $|t| < 2-\sqrt{3}$
and \eqref{eq:bound1_using_sides}, we have
$$
\begin{aligned}
\HHH^2(\superficiesteiner & \cap \{-h\tau < x < h\tau\})
\\
\geq &
 \int_{-h\tau}^{h\tau} \HHH^1(\genericsurface \cap \pi_t) ~dt
\\
 \geq &
  2h \int_0^s ( 2 - 2t ) ~dt
  + 2h \int_s^\tau \left[ 1 + \sqrt{3} - (\sqrt{3} - 1) t \right] ~dt
\\
 = &
  h[(2+2\sqrt{3})\tau - (\sqrt{3}-1)\tau^2 - (2\sqrt{3}-2)s - (3-\sqrt{3})s^2],
\end{aligned}
$$
that has to be compared with
$$
\HHH^2(\particularsurfaceSteiner \cap \{-h\tau \leq x \leq h\tau\}) =
 (2 + \tau) \tau \sqrt{4 h^2 + 1} + 3.
$$
The only difference with the derivation of Theorem \ref{teo:comparison} is the presence of the
two terms containing the parameter $s$.
They however vanish as $s \to 0^+$, so that by selecting $s>0$ sufficiently small we can again conclude if
$h$ is sufficiently large.
\end{proof}

%

Specific values for $s_0$ and $h$ turn out to be
$$
s_0 = \frac{2 - \sqrt{3}}{4}, \qquad h \geq 40
$$
or
$$
s_0 = \frac{2 - \sqrt{3}}{100}, \qquad h \geq 16
$$

The result of the previous theorem suggests the following

\begin{conjecture}\label{conj:NW}
Let $h > 16$ and $s = \frac{2 - \sqrt{3}}{100}$. Then $\basehs$ satisfies condition (NW) of
Definition \ref{def:NW}.
\end{conjecture}

\section{All possible triple covers}\label{sec:allcoverings}
We want to describe all possible covers of the base space $\base$ of degree $3$
among those that will produce soap films that touch all six edges of the wedge and
are not forced to touch the invisible circular wires.
Here the fundamental group comes obviously into play, since 
we do that by describing all possible monodromy actions on the fiber above the
base point $\basepoint$.
This monodromy action can equivalently be described as the action defined by a
subgroup of $\pi_1(B)$ of index $3$ by right multiplication.

\begin{remark}
The presence of triple points in the wireframe of the tetrahedron,
together with the wetting condition (implying the existence of triple lines in the minimizing
film) requires the presence of at least two distinct and nontrivial monodromy actions on the
fiber at infinity, hence the cover has at least degree $3$.
\end{remark}

\begin{remark}
A degree $3$ cover cannot allow for the standard (i.e., conelike) minimizing film for the tetrahedron.
This is due to the presence of the central quadruple point of the minimizing film.
Indeed we shall see that all covers satisfying the constraints will require a nontrivial
monodromy action on the fiber when circling the invisible wires, making them an essential
feature in the construction of the base space. 
Of course we could allow for more than two invisible wires.
\end{remark}

Let $H$ be a subgroup of $\pione$ of index $3$, and denote by 
$s_1 := H$, $s_2 := H w_2$ and $s_3 := H w_3$ 
its right cosets,
for some choice of representative
elements $w_2, w_3 \in \pione$.
The elements of $\pione$ define an action on $\{s_1, s_2, s_3\}$ defined by
$g : h \to hg$ (right multiplication by $g \in \pi_1(B)$).
If $h'$, $h''$ are in the same right coset then $h' g (h'' g)^{-1} = h' (h'')^{-1}$
and the definition is wellposed.

We then have a map $\pione \to S_3$ that to any element of $\pione$ associates
an element of the permutation group of the three cosets.
A permutation of $S_3$ is interpreted as a permutation of the indices in
$\{s_1, s_2, s_3\}$.

By composition, this map is defined once we know its value on the generators of
$\pione$.

In conclusion we have permutations $\sigma_a$,
$\sigma_b$,
$\sigma_c$,
$\sigma_d$,
$\sigma_e \in S_3$
(permutations of $\{1,2,3\}$) associated to the generators
$a, b, c, d, e$ respectively.

We now impose a number of constraints.

\begin{description}
\item[Consistency with relators]
In the group presentation \eqref{eq:presentation} the two relators must be consistent
with the choice of the permutations $\sigma_a, \dots, \sigma_e$.
In other words $\sigma_a$ must commute with $\sigma_b$ and
$\sigma_d$ must commute with $\sigma_e$.
Take for example $\sigma_a$ and $\sigma_b$; they commute if and only if one of the following
mutually exclusive conditions holds:
\begin{enumerate}
\item
$\sigma_a$ or $\sigma_b$ is the identity permutation $()$;
\item
$\sigma_a$ and $\sigma_b$ are both cyclic of order 3, hence a power
of $(1,2,3)$;
\item
$\sigma_a = \sigma_b$ are the same transposition.
\end{enumerate}
\item[Invisible-wire conditions]
The soap film  that we wish to model must not wet the two circular loops associated
to generators $a$ and $e$ in Figure \ref{fig:covering_base}.
In particular, a closed path starting at the base point in $\base$ and looping around one
of such loops must not necessarily traverse the surface.
Consequently the generators $a$ and $b$ must not move sheet $1$ of the covering.
The corresponding condition reads then as
\begin{equation}\label{eq:invcond}
\sigma_a, \sigma_e \in \{ (), (2,3) \} ;
\end{equation}
\item[Wetting conditions]
We want to reconstruct a film  that spans all six sides of the wedge.  In other
words, any tight loop around these edges should cross the surface.
This condition is somewhat tricky to impose, particularly in situations where the cover
is not normal, because we need to state it on elements of $\pione$, which
requires to connect the base point to the tight loop.
We end up with a condition that depends on \emph{how} we choose the connecting path.
Changing the path amounts to performing a conjugation on the element of $\pione$.
A strong wetting condition could be that any element in $\pione$ that loops once around
the selected edge must move all sheets (the permutation is required to be a derangement).
This condition is insensitive to conjugation.
We shall however require a
weaker version of the wetting condition by requiring that the element of $\pione$ moves sheet $1$ (the sheet
where $u = 1$ at the base point $\basepoint$, far from $\tetrahedron$).
This condition however depends on how we connect the tight loop to the base point.
It seems only natural to require the connecting path to lie outside the wedge, which is
not the same as requiring the corresponding Wirtinger-type loop in the diagram of Figure
\ref{fig:covering_base} to move sheet $1$.
In particular this is not true for $\longthree$, the long edge that in Figure \ref{fig:covering_base}
runs in the back and does not cross the two disks, for which a linking path that does not enter
the wedge is $bc^{-1}$.
In the end the (weak) wetting conditions read as:
\begin{equation}\label{eq:wetcond}
\begin{aligned}
&\sigma_c, \quad \sigma_c^{-1} \sigma_d, \quad \sigma_b \sigma_c^{-1}, \quad
 \sigma_b \sigma_c^{-1} \sigma_d,
\\
&\sigma_a \sigma_d^{-1} \sigma_c \sigma_a^{-1} \sigma_c^{-1}, \quad
 \sigma_b \sigma_c^{-1} \sigma_e \sigma_c \sigma_e^{-1} \not\in \{ (), (2,3) \}
\end{aligned}
\end{equation}
The third relation comes from $\sigma(\longs_{4,2})^{-1} \sigma(\longs_{3,3}) \sigma(\longs_{4,2})$,
where $\sigma(\longs_{i,j})$ is the permutation associated to the varius long edges in \eqref{eq:arcs},
after substitution and simplification.
It corresponds to a path that, as mentioned above, starts at $\basepoint$, runs in the back of $\longfour$ from
below, than around $\longthree$, than again in the back of $\longfour$ and back to $\basepoint$.
The fourth relation is the inverse of $\sigma(\longs_{4,2})$.
\end{description}

We shall now search for all possible choices of 
$\sigma_a, \sigma_b, \sigma_c, \sigma_d, \sigma_e$ that are compatible with
the three set of constraints \eqref{eq:invcond}, \eqref{eq:wetcond} and
consistency with the relators of the presentation.

\begin{remark}
Our search also includes the special cases were one or both of the invisible wires $\loopone$ and $\looptwo$
are not present, since the choice, say, $\sigma_a = ()$ 
leads to the same result
as removal of $\loopone$.
\end{remark}

We separately analize the possibilities with all choices of $\sigma_a$ and $\sigma_e$ allowed
by \eqref{eq:invcond} arriving to the conclusion that the presence of both invisible wires is essential.

\subsection{Searching covers for \texorpdfstring{$\sigma_a = \sigma_e = ()$}{a=e=()}}
First note that all constraints above are insensitive to exchange of sheets $2$ and $3$.
This means that for definiteness we can assume that $\sigma_c(1) = 2$, which in view of
the first wetting constraint in \eqref{eq:wetcond} leaves us with only two possibilities:
$\sigma_c = (1,2)$ or $\sigma_c = (1,2,3)$.

\subsubsection{Case \texorpdfstring{$\sigma_c = (1,2,3)$}{c=(123)}}
Wetting constraints $2, 3, 5, 6$ imply $\sigma_d \in \{(1,3,2),(1,2)\}$ resulting
in $\sigma_c^{-1} \sigma_d$ sending $3 \mapsto 1$.
Moreover $\sigma_b \in \{(1,3,2),(1,3)\}$ resulting in $\sigma_b$ sending
$1 \mapsto 3$.
This would imply $\sigma_b \sigma_c^{-1} \sigma_d$ sending $1 \mapsto 1$, contrary to
wetting constraint $4$.

\subsubsection{Case \texorpdfstring{$\sigma_c = (1,2)$}{c=(12)}}
Wetting constraints $2, 3, 5, 6$ imply $\sigma_d \in \{(1,2,3),(1,3)\}$ resulting
in $\sigma_c^{-1} \sigma_d$ sending $3 \mapsto 1$.
Moreover $\sigma_b \in \{(1,3,2),(1,3)\}$ resulting in $\sigma_b$ sending
$1 \mapsto 3$.
This would imply $\sigma_b \sigma_c^{-1} \sigma_d$ sending $1 \mapsto 1$, contrary to
wetting constraint $4$.

\subsection{Searching covers for \texorpdfstring{$\sigma_a = ()$ and $\sigma_e = (2,3)$}{e=(23)}}
Again we can assume that $\sigma_c(1) = 2$.

Reasoning as before, from $\sigma_a = ()$ we get $\sigma_d \in \{(1,3,2),(1,2,3),(1,2),(1,3)\}$.
However $\sigma_d$ and $\sigma_e$ must commute, which is incompatible with $\sigma_e = (2,3)$.

\subsection{Searching covers for \texorpdfstring{$\sigma_a = (2,3)$ and $\sigma_e = ()$}{a=(23)}}
Again we can assume that $\sigma_c(1) = 2$.

Reasoning as before from $\sigma_e = ()$ we get $\sigma_b \in \{(1,3,2),(1,2,3),(1,2),(1,3)\}$,
which does not commute with $\sigma_a$.

\subsection{Searching covers for \texorpdfstring{$\sigma_a = \sigma_e = (2,3)$}{(a=e=(23)}}
Consistency with the relators of the group presentation leads to
$$
\sigma_b, \sigma_d \in \{ (), (2,3) \} , \qquad
\sigma_c \not\in \{ (), (2,3) \} .
$$
A direct check shows that any choice satisfying the requirements 
above also satisfies all
constraints for our covering.

\begin{figure}
\includegraphics[width=0.7\textwidth]{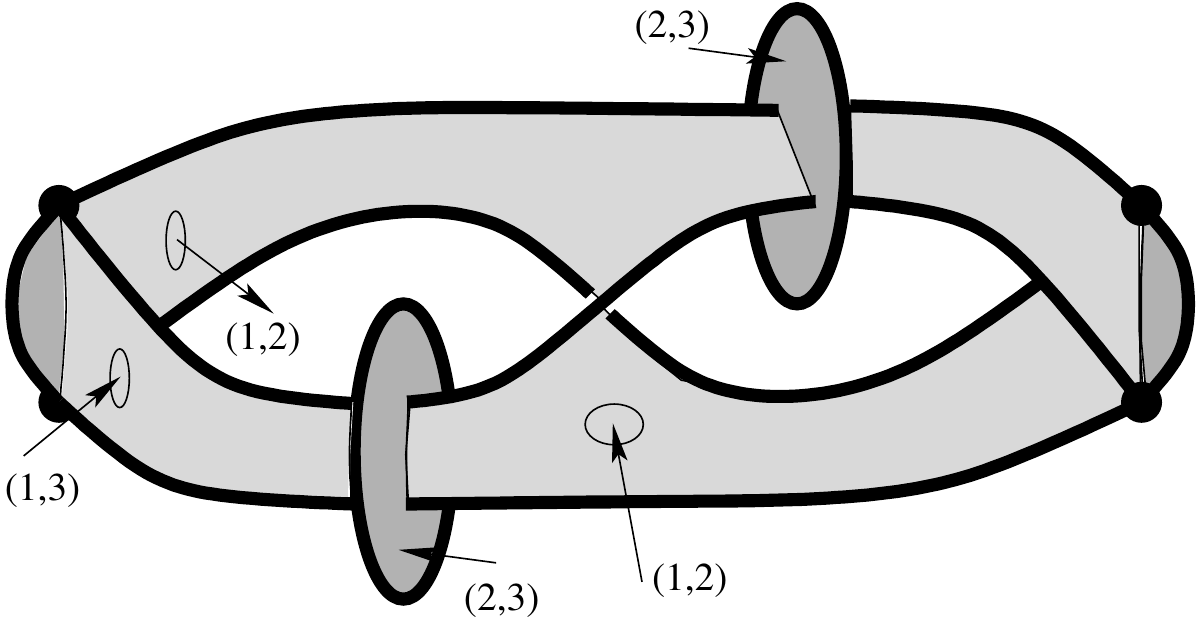}
\caption{\small{An alternative covering definition that should lead to the same
minimizing film for the covering of Figure \ref{fig:covering_cut}.}}
\label{fig:covering2_cut}
\end{figure}

Figure \ref{fig:covering2_cut} shows the cover that corresponds to the choice
$$
\sigma_b = \sigma_d = (), \qquad \sigma_c = (1,2).
$$
The resulting cover is clearly not isomorphic to the one constructed in Section \ref{sec:covering}
and it is a natural question whether the minimization problem $\FFF(u) \to \text{min}$ in this
context leads to the same result.
This is entirely possible but we do not want to pursue the subject here.

On the other hand we actually can find functions $u \in \domainF$ (with $\domainF$ redefined based on
the new covering) with a jump set that is incompatible with the former definition of $\domainF$.
We can construct such a function by making it jump across a big sphere that encloses $\tetrahedron$
with the sheet where its value is $1$ that changes from $1$ to $3$.
Then it is possible to take advantage of the fact that the wetting conditions are weak and with
$u = 1$ on sheet $3$ they do not impose wetting of e.g. $\longs_{1,1}$ resulting in the possibility
(actually achievable) of an only partially wetted $\longone$.
Similarly for the other long edges.
Clearly, however, such a surface cannot be a minimizer.

\section{Numerical simulations and conclusions}\label{sec:conclusions}
A number of numerical simulations have been performed using the software code
\texttt{surf} of Emanuele Paolini.
It is based on a gradient flow with artificial viscosity starting from a triangulated
surface having the required topology.
It does not use the setting based on coverings of the present paper, however it gives a
consistent result provided that:
(i) the starting surface is the set $\projectionofjumpofu$ of
some $u \in \domainF$; (ii) there is no change of topology; (iii) there is no touching
of the invisible wires (they are not modelled by \texttt{surf}).
Figures \ref{fig:surfacefilmskew}, \ref{fig:surfacefilm} and \ref{fig:surfacefilmc} have
all been obtained by starting from suitable faceted initial surfaces, with the geometry
corresponding to the choice $h = 3.5$; for example the result shown in Figure
\ref{fig:surfacefilm} is obtained by starting from the faceted surface displayed
in Figure \ref{fig:surface2}.

Numerically it turns out that with $h=3.5$ the area of the non-simply connected
minimizer is slightly
greater than the area of the conelike configuration; on the contrary
increasing $h$ to (e.g.) $h=4$ results in a non-simply connected 
film surface that numerically beats
the conelike configuration, consistently with the results of Section \ref{sec:comparison}.

Decreasing $h$ changes the minimizing evolution drastically: after a (large) number of
gradient flow iterations, the film surface loses its symmetry (due to roundoff errors that
break the symmetry of the problem) and one of the two 
tunnels shrinks at the expense
of the other.
The numerical evolution stops when the smaller hole completely closes, since the software
cannot cope with changes of topology.
Evolution after such singularization time depends on how the topology is modified.
However it should be noted that the evolution would in this case typically impact with
one of the invisible wires before the singularization time.

It is conceivable that for this value of $h$ the evolution would produce a stationary
surface, that is area-minimizing
among surfaces that are forced to have the same symmetries
of the boundary frame.

Decreasing $h$ even further, in particular towards the value $h = \frac{\sqrt{2}}{2}$ that results in a regular
tetrahedron, numerically produces an evolution where the two tunnels both shrink more or less
selfsimilarly, so that we expect in the limit to obtain the area-minimizing cone
of Figure \ref{fig:morgan} left.

On the other side we can explore what happens with ever increasing values of $h$ (and a fixed sufficiently
small value of $s$).
Figure \ref{fig:surfacefilmhlarge} shows the numerical solution for $h = 20$, where the $x$ coordinate has been
shrunk down in order to fit the same frame $\Wedge$ of Figure \ref{fig:surface2}.
The resemblance of the result with the constructed surface $\particularsurface$
of Figure \ref{fig:surface2} is striking and
suggest to conjecture that a minimizer $\superficieminima \in \superficiminime(\basehs)$,
when rescaled appropriately in the direction
orthogonal to the short sides in order to have a fixed boundary, converges to $\particularsurface$ as
$h \to +\infty$.
This fact can be also motivated by observing that the area of a surface that is deformed by scaling
of a factor $k$ in the $x$ direction can be computed by using an anisotropic version of the area functional
$\int \phi(\nu) ~d \HHH^2$ with $\phi$ a positive one-homogeneous function having
as unit ball the set $\{k^2 x^2 + y^2 + z^2 \leq 1\}$, and $\nu$ 
a unit normal vector field.
With increasing values of $k$ the ``vertical'' portions of a surface (those with local constant $x$
coordinate) pay less and less in anisotropic area and we expect that in the limit $k \to +\infty$
the anisotropic area to simply be given by the integral in $x$ of the $\HHH^1$ measure of the
sections of the rescaled surface with vertical planes
parallel to the $yz$ coordinate plane, so that a minimizer can be obtained by separately minimizing
the size of each section.
This would essentially lead to the faceted surface $\particularsurface$ of Figure \ref{fig:surface2}.

\begin{figure}
%
%
\begin{overpic}[width=0.9\textwidth]{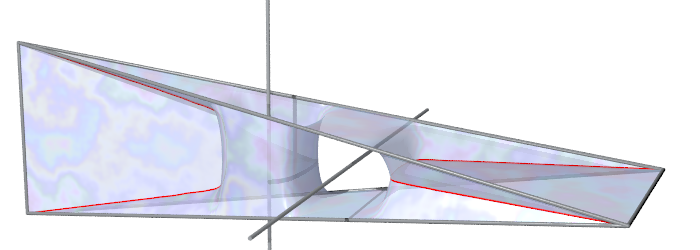}
 \put (60,30) {\large [$x$ scaled $0.175$]}
\end{overpic}
\caption{\small{Minimal film obtained for $h=20$.  The resulting surface is scaled down in the $x$ (right-left)
direction of a factor of $3.5/20$ in order to match with the frame for $h=3.5$ of Figure \ref{fig:surface2}.}}
\label{fig:surfacefilmhlarge}
\end{figure}



\end{document}